\documentclass[a4paper,10pt,oneside]{article}

\usepackage{tikz}
\usetikzlibrary{matrix,arrows,decorations.pathmorphing}
\usepackage{pdfpages}
\usepackage[english]{babel}
\usepackage[T1]{fontenc}
\usepackage{lmodern}
\usepackage[top=3cm,bottom=4cm,right=1.5cm,left=1.5cm]{geometry}
 \usepackage{textcomp} 
 \usepackage{amsmath, amssymb} 
 \usepackage{amsthm}
 \usepackage{graphicx}
 \usepackage{color}

 \usepackage{lmodern}
\renewcommand{\baselinestretch}{1.2}

\usepackage{etoolbox}

\usepackage{hyperref}

\hypersetup{	
     colorlinks=true,       
      breaklinks=true,         
      urlcolor= magenta,       
      linkcolor=blue ,	
      citecolor=blue,	
                       }

\slshape
\newcommand{\R}{\mathbb R }

\newcommand{\N}{\mathbb N}
\newcommand{\C}{\mathbb C}

\theoremstyle{definition}

 \newtheorem{definition}{Definition :}[section]
  
\newtheorem{theo}{Theorem}[section]
   \newtheorem{pro}{Proposition}[section]
    
     \newtheorem{rem}{Remark}[section]
     
     \newtheorem{lem}{Lemma}[section]
     
     \newtheorem{cor}{Corollary}[section]

\numberwithin{equation}{section}

\begin{document}
\title{A characterization of  the  $L^2$-range of the Poisson transforms on a class of vector bundles over the quaternionic hyperbolic spaces}

\author{
Abdelhamid Boussejra \thanks{e-mail: boussejra.abdelhamid@uit.ac.ma}
Achraf Ouald Chaib\thanks{e-mail:achraf.oualdchaib@uit.ac.ma} \\
\begin{small}
Department of Mathematics,  Faculty of Sciences
\end{small}\\
\begin{small}
University Ibn Tofail, K\'enitra, Morocco
\end{small}}

\date{}
\maketitle
\begin{abstract}
We study the $L^2$-boundedness of the Poisson transforms associated to the   homogeneous vector bundles \\$Sp(n,1)\times_{Sp(n)\times Sp(1)} V_\tau$ over the quaternionic hyperbolic spaces $Sp(n,1)/Sp(n)\times Sp(1)$ associated  with irreducible representations $\tau$ of $Sp(n)\times Sp(1)$ which are trivial on $Sp(n)$. As a consequence, we describe the image of the section space $L^2(Sp(n,1)\times_{Sp(n)\times Sp(1)} V_\tau)$ under the generalized spectral projections associated to  a family of  eigensections of the Casimir operator.
\end{abstract}
\textbf{Keywords}: Vector Poisson transform,  Fourier restriction estimate, Strichartz conjecture.
\section{Introduction}
Let $G$ be 	 a connected real semisimple noncompact Lie group with finite center, and $K$ a maximal compact subgroup. Then $X=G/K$ is a Riemannian symmetric space of noncompact type.  Let $G=KAN$ be an Iwasawa decomposition of $G$, and let $M$ be the centralizer of $A$ in $K$. We write $g=\kappa(g){\rm e}^{H(g)}n(g)$, for each $g\in G$ according to $G=KAN$. 
A central result in harmonic analysis (see \cite{Ka}) asserts that all joint eigenfunctions $F$ of the algebra $\mathbb{D}(X)$ of invariant differential operators, are  Poisson integrals  
$$
F(g)=\mathcal{P}_\lambda f(g):= \int_K {\rm e}^{(i\lambda+\rho)H(g^{-1}k)}f(k)\,{\rm d}k,
$$ 
of a hyperfunction $f$ on $K/M$, for a generic $\lambda\in \mathfrak{a}_c^\ast$ (the complexification of $\mathfrak{a}^\ast$ the real dual of $\mathfrak{a}$).\\
Since then a  characterization of the $L^p$-range of the Poisson transform  was developed in several articles such as  \cite{BI}, \cite{BS}, \cite{B}, \cite{BO}, \cite{I},   \cite{KR}, \cite{KU}, \cite{L}, \cite{SJ},  \cite{S}.\\
The problem of characterizing the image of the Poisson transform $\mathcal{P}_\lambda$ of $L^2(K/M)$ with real and regular spectral parameter $\lambda$ is intimately related to Strichartz conjecture [\cite{S}, Conjecture 4.5] on the uniform $L^2$-boundedness of the generalized spectral projections associated with $\mathbb{D}(X)$.
 To be more specific, consider the generalized spectral projections $\mathcal{Q}_\lambda$ defined initially for $F\in C_c^\infty(X)$ by
\begin{align}\label{spe}
\mathcal{Q}_\lambda F(x)=\mid \mathbf{c}(\lambda)\mid^{-2}\mathcal{P}_\lambda(\mathcal{F}F(\lambda,.)(x), \quad \lambda\in \mathfrak{a}^\ast,
\end{align}
where $\mathcal{F}F$ is the Helgason Fourier transform of $F$ and $\mathbf{c}(\lambda)$ is the Harish-Chandra $c$-function.\\
\textbf{Conjecture} (Strichartz [\cite{S}, Conjecture 4.5]).
There exists a positive constant $C$ such that for any $F_\lambda=\mathcal{Q}_\lambda F$ with $F\in L^2(X)$ we have 
\begin{align}\label{conj1}
C^{-1}\parallel F\parallel^2_{L^2(X)}\leq \sup_{R>0,y\in X}\, \int_{\mathfrak{a}_+^\ast}\frac{1}{R^r}\int_{B(y,R)}\mid F_\lambda(x)\mid^2\, {\rm d}x\, {\rm d}\lambda\leq C \parallel F\parallel^2_{L^2(X)},
\end{align}
and 
\begin{align}\label{conj2}
\parallel F\parallel^2_{L^2(X)}=\gamma_r\lim_{R\rightarrow \infty} \int_{\mathfrak{a}_+^\ast}\frac{1}{R^r}\int_{B(y,R)}\mid F_\lambda(x)\mid^2\, {\rm d}x\, {\rm d}\lambda.
\end{align}
Conversely, if $F_\lambda$ is any family of joint eigenfunctions for which the right hand side of \eqref{conj1} or \eqref{conj2} is finite, then there exists $F\in L^2(X)$ such that $F_\lambda=\mathcal{Q}_\lambda F$ for a.e. $\lambda\in \mathfrak{a}_+^\ast$. \\
Here $r=\textit{rank}\, X$, and $B(y,R)$ denotes the open ball in $X$ of radius $R$ about  $y$. The constant $\gamma_r$ depends on the normalizations of the measures ${\rm d}x$ and ${\rm d}\lambda$.\\
The strichartz conjecture has been recently  settled by Kaizuka, see \cite{KK}.
Most of the proof consists in proving a uniform estimate for the Poisson transform.
More precisely, the following was proved by Kaizuka [\cite{KK}, Theorem 3.3]:\\
Let $F$ be a  joint eigenfunction with eigenvalue corresponding to a real and regular spectral parameter $\lambda$ . Then  
\(F\) is the Poisson transform by $\mathcal{P}_\lambda$ of some \(f\in L^2(K/M)\) if and only if
\begin{align*}
\sup_{R>1}\frac{1}{R^r}\int_{B(0,R)}\mid F(x)\mid^2\,{\rm d}x<\infty. 
\end{align*}
Moreover there exists a positive constant \(C\) independent  of such  \(\lambda\),
\begin{equation*}
C^{-1}\mid \mathbf{c}(\lambda)\mid^2 \parallel f\parallel^2_{L^2(K/M)}\leq \sup_{R>1}\, \frac{1}{R^r}\int_{B(0,R)}\mid \mathcal{P}_\lambda f(x)\mid^2\,{\rm d}x\leq C\mid \mathbf{c}(\lambda)\mid^2 \parallel f\parallel^2_{L^2(K/M)}.
\end{equation*}
The generalization of these results to vector bundles setting has only just begin. In \cite{BIA} we extend Kaizuka result to homogeneous line bundles over non-compact complex Grassmann manifolds (See also \cite{BI1}).\\Our aim in this paper is to generalize theses results to a class of  homogeneous vector bundles over the quaternionic hyperbolic space \(G/K\), where \(G\) is the symplectic group \(Sp(n,1)\) with maximal compact subgroup \(K=Sp(n)\times Sp(1)\).\\
To state our results in rough form, let us first introduce the class of the homogenous vector bundles that we  consider in this paper. Let \(\tau_\nu\) be a unitary irreducible representation of \(Sp(1)\) realized on a $(\nu+1)$-dimensional Hilbert  space $(V, (.,\, .)_\nu)$. We extend $\tau_\nu$ to a representation of $K$ by setting $\tau_\nu\equiv 1$ on $Sp(n)$. 
As usual the space of sections of the homogeneous vector bundle $G\times_K V$  associated with $\tau_\nu$ will be identified with the space $\Gamma(G,\tau_\nu)$ of vector valued functions  $\displaystyle F: G\rightarrow V_\nu$ which  are right $K$-covariant of type $\tau_\nu$, i.e.,
\begin{align}\label{cov}
F(gk)=\tau_\nu(k)^{-1}F(g), \quad \forall g\in G, \quad \forall k\in K.
\end{align}
We denote by $C^\infty(G,\tau_\nu)$ and  $C_c^\infty(G,\tau_\nu)$  the elements of $\Gamma(G,\tau_\nu)$ that are respectively smooth, smooth with compact support in $G$, and by $L^2(G,\tau_\nu)$ the elements of $\Gamma(G,\tau_\nu)$ such that 
\begin{align*}
\parallel F\parallel_{L^2(G,\tau_\nu)}=\left(\int_{G/K} \parallel F(g)\parallel^2_\nu\, {\rm d}g_K\right)^{\frac{1}{2}}<\infty.
\end{align*}
In above $\parallel .\,\parallel_\nu$ is the norm in $V_\nu$ and $\parallel F(gK)\parallel_\nu=\parallel F(g)\parallel_\nu$ is well defined for $F$ satisfying \eqref{cov}.\\
Let $\sigma_\nu$ denote the restriction of $\tau_\nu$ to the group $M\simeq Sp(n-1)\times Sp(1)$. Over $K/M$ we have the associated homogeneous vector bundle $K\times _M V_\nu$ with $L^2$-sections identified with $L^2(K,\sigma_\nu)$  the space  of all functions $f:K\rightarrow V_\nu$ which are $M$-covariant of type $\sigma_\nu$ and satisfy
\begin{align*}
\parallel f\parallel^2_{L^2(K,\sigma_\nu)}=\int_K \parallel f(k)\parallel_\nu^2 {\rm d}k<\infty,
\end{align*}
where ${\rm d}k$ is the normalized Haar measure of $K$.\\
For $\lambda\in \mathbb{C}$ and $f\in L^2(K,\sigma_\nu)$, the Poisson transform $\mathcal{P}^\nu_\lambda f$    is defined by 
\begin{align*}
\mathcal{P}^\nu_\lambda f(g)=\int_K {\rm e}^{-(i\lambda+\rho)H(g^{-1}k)}\tau_\nu(\kappa(g^{-1}k))f(k)\,{\rm d}k
\end{align*}  
Let $\Omega$ denote  the Casimir element  of the Lie algebra $\mathfrak{g}$ of $G$, viewed as a differential operator acting on $C^\infty(G,\tau)$. Then the image $\mathcal{P}^\nu_\lambda(L^2(K,\sigma_\nu))$ is a proper closed subspace of $\mathcal{E}_\lambda(G,\tau_\nu)$ the space of all $F\in C^\infty(G,\tau_\nu)$ satisfying 
\begin{align*}
\Omega\,F=-(\lambda^2+\rho^2-\nu(\nu+2))F.
\end{align*}
For more details see section 2.\\
For $\lambda\in \mathbb{R}\setminus\{0\}$, we define a weighted $L^2$-space $\mathcal{E}_\lambda^2(G,\tau_\nu)$ consisting  of all   $F$ in $\mathcal{E}_{\lambda}(G,\tau_\nu)$  that satisfy 
\begin{align*}
\parallel F\parallel_\ast=\sup_{R > 1}\left(\frac{1}{R}\int_{B(R)}\|F(g)\|_\nu^{2}\, {\rm d}g_K\right)^{\frac{1}{2}}<\infty.
\end{align*}
Our first main result is an image characterization of the Poisson transform $\mathcal{P}_\lambda^\nu$ of $L^2(K,\sigma_\nu)$ for $\lambda\in \mathbb{R}\setminus \{0\}$.
\begin{theo}\label{range}
Let  \(\lambda \in \mathbb{R} \backslash\{0\}$ and  $\nu$ a nonnegative integer.
\begin{enumerate}
\item[(i)] There exists a positive constant \(C_\nu\) independent of \(\lambda\) such that for $f\in L^2(K, \sigma_\nu)$ we have
\begin{equation}\label{Poisson estimates}
C_{\nu}^{-1}|\mathbf{c}_\nu(\lambda)|\, \|f\|_{L^{2}(K,\sigma_\nu)}\leq \|\mathcal{P}_\lambda^\nu f\|_\ast\leq C_{\nu}|\mid \mathbf{c}_\nu(\lambda)\mid \, \|f\|_{L^{2}(K,\sigma_\nu)},
\end{equation}
with
\[
\mathbf{c}_\nu(\lambda)=2^{\rho-i\lambda}\frac{\Gamma(\rho-1)\Gamma(i\lambda)}{\Gamma(\frac{i\lambda + \nu +\rho}{2})\Gamma(\frac{i\lambda+\rho-\nu-2}{2})}.
\]
Furthermore we have the following Plancherel type formula for the Poisson transform
\begin{equation}\label{estimate00}
\lim_{R\rightarrow +\infty}\dfrac{1}{R}\int_{B(R)}\Vert \mathcal{P}_\lambda^\nu f(g)\Vert_\nu^{2}\,{\rm d}g_K=2\mid \mathbf{c}_\nu(\lambda)\mid^2\left\|f\right\|_{L^{2}(K,\sigma_\nu)}^2.
\end{equation}
\item[ii)] $\mathcal{P}_\lambda^\nu$ is a topological isomorphism from \(L^2(K,\sigma_\nu)\) onto \(\mathcal{E}^2_\lambda(G,\tau_\nu)\).
\end{enumerate}
\end{theo}
This generalizes the result of Kaizuka [\cite{KK}, (i) and (ii) in Theorem 3.3] which corresponds to \(\tau_\nu\) trivial.\\

\textbf{Consequence}\\
For $\lambda\in \mathbb{R}$ we define the space
$$
\mathcal{E}^\ast_{\lambda}(G,\tau_\nu)=\{F\in \mathcal{E}_\lambda(G,\tau_\nu): M(F)<\infty\},
$$
where
$$
M(F)=\lim\sup_{R\rightarrow \infty}\left(\frac{1}{R}\int_{B(R)}\parallel F(g)\parallel_\nu^2\, {\rm d}g_K\right)^\frac{1}{2}.
$$
Then as an immediate consequence of Theorem \ref{range} we obtain the following result which  generalizes a conjecture of W. Bray \cite{Br} which corresponds to $\tau_\nu$ trivial.
\begin{cor}
If $\lambda\in \mathbb{R}\setminus\{0\}$ then $\mathcal{E}^\ast_{\lambda}(G,\tau_\nu),M)$ is a Banach space.
\end{cor}
\begin{rem}
In the case of the trivial bundle (the scalar case)  the conjecture of Bray was proved  by Ionescu \cite{I} for all rank one symmetric spaces . It was generalized to Riemannian symmetric spaces of higher rank by Kaizuka, see \cite{KK}.
\end{rem} 
Next, let us introduce our second main result on the $L^2$-range of the generalized spectral projections.\\
For $F\in C^\infty_c(G,\tau_\nu)$ the vector valued Helgason-Fourier  transform $\mathcal{F}_\nu F$   is given by (see \cite{com})
$$
\mathcal{F}_\nu\, F(\lambda,k)=\int_G {\rm e}^{(i\lambda-\rho)H(g^{-1}k)}\tau_\nu(\kappa(g^{-1}k)^{-1})   F(g)\,{\rm d}g  \quad \lambda\in \mathbb{C},
$$
Then the following inversion formula holds  (see section 4)
\begin{equation}\begin{split}\label{inversion}
\quad F(g)=\frac{1}{2\pi}&\int_0^\infty\int_K {\rm e}^{-(i\lambda+\rho)H(g^{-1}k)}\tau_\nu(\kappa(g^{-1}k))\mathcal{F}_\nu F(\lambda,k)\,\mid \mathbf{c}_\nu(\lambda)\mid^{-2}{\rm d}\lambda\,{\rm d}k\\
&+\sum_{\lambda_j\in D_\nu} d_\nu(\lambda_j)\int_K {\rm e}^{-(i\lambda_j+\rho)H(g^{-1}k)}\tau_\nu(\kappa(g^{-1}k))\mathcal{F}_\nu F(\lambda_j,k)\,{\rm d}k.
\end{split}\end{equation} 
In above $\displaystyle d_\nu(\lambda)=-i\textit{Res}_{\mu=\lambda}(\mathbf{c}_\nu(\mu)\mathbf{c}_\nu(-\mu))^{-1}, \lambda\in D_\nu $ and $D_\nu$ is a finite set in $\{\lambda\in \mathbb{C}; \Im(\lambda)>0\}$ which parametrizes 
the  $\tau_\nu$-spherical functions arising from the discrete series of $G$. It is empty if $\nu\leq \rho-2$.\\
The formula \eqref{inversion} gives rise to the decomposition of $L^2(G,\tau_\nu)$ into a continuous part and a discrete part:
\begin{align*}
L^2(G,\tau_\nu)=L_{\textit{cont}}^2(G,\tau_\nu)\oplus L_{\textit{disc}}^2(G,\tau_\nu) 
\end{align*}
Our aim here is to study the operator $\mathcal{Q}^\nu_\lambda$, $\lambda\in \mathbb{R}$, defined for $F\in L_{\textit{cont}}^2(G,\tau_\nu)\cap C_c^\infty(C,\tau_\nu)$ by 
\begin{align}\label{specv}
\mathcal{Q}^\nu_\lambda F(g)=\mid \mathbf{c}_\nu(\lambda)\mid^{-2}\mathcal{P}^\nu_\lambda[\mathcal{F}_\nu\, F(\lambda,.)](g),
\end{align}
More precisely, following Strichartz idea, we are interested in the following question:\\ Characterize those $F_\lambda\in \mathcal{E}_\lambda(G,\tau_\nu)$ ($\lambda\in (0,\infty)$) for which there exists $F\in L_{\textit{cont}}^2(G,\tau_\nu)$ such that $F_\lambda=\mathcal{Q}^\nu_\lambda F$.\\
To do so, we introduce  the space $\mathcal{E}^2_+(G,\tau_\nu)$ consisting of all $V_{\tau_\nu}$-valued measurable functions $\psi$ on $(0,\infty)\times G$ such that 
\begin{enumerate}
\item[(i)] $\Omega\, \psi(\lambda,.)=-(\lambda^2+\rho^2-\nu(\nu+2))\,\psi(\lambda,.)$ a.e. $\lambda\in (0,\infty)$
\item[(ii)] $\parallel \psi\parallel_+<\infty$.
\end{enumerate}
where
$$
\parallel \psi\parallel^2_+=\sup_{R>1}\int_0^\infty\frac{1}{R}\int_{B(R)}\parallel\psi(\lambda,g) \parallel^2_\nu\,{\rm d}g_K\,{\rm d}\lambda.
$$
The second main result we prove in this paper can be stated as follows 
\begin{theo}\label{spectre}
\begin{enumerate}
\item[(i)] There exists a positive constant $C$ such that for $F\in L^2(G, \tau_\nu)$ we have
\begin{align}\label{spectral2}
C^{-1}\parallel F\parallel_{L^2(G,\tau_\nu)}\leq \parallel \mathcal{Q}^\nu_\lambda F\parallel_+\leq C\parallel F\parallel_{L^2(G,\tau_\nu)}
\end{align}
Furthermore we have 
\begin{align}\label{spectral3}
\lim_{R\rightarrow \infty}\int_0^\infty\frac{1}{R}\int_{B(R)}\parallel \mathcal{Q}^\nu_\lambda F(g) \parallel^2_\nu\,{\rm d}g_K\,{\rm d}\lambda=2\parallel F\parallel^2_{L^2(G,\tau_\nu)}
\end{align}
\item[(ii)] The linear map $\mathcal{Q}^\nu_\lambda $  is a topological isomorphism from $ L_{\textit{cont}}^2(G,\tau_\nu)$ onto  $\mathcal{E}^2_+(G,\tau_\nu)$.
\end{enumerate}
\end{theo}
This extends Kaizuka result [ \cite{KK}, (i) and (ii) in  Theorem 3.6]   on the Strichartz conjecture (see  \cite{S} Conjecture 4.5] to the class of vector bundles considered here.\\
Before giving the outline of the paper, let us mention that  a number of authors have obtained an image characterization for the Poisson transform $ \mathcal{P}_\lambda$ ($\lambda\in \mathfrak{a}^\ast\setminus \{0\}$)  of $L^2$-functions on $K/M$ in the rank one case, see  [\cite{BI}, \cite{BS}, \cite{BO}, \cite{I}]. Nevertheless, the obtained characterization  is  weaker than the one conjectured by Strichartz. The approach taken in the quoted papers is based on the theory of Calderon-Zygmund singular integrals (see also  \cite{KU}). Using a different approach based on the techniques used in the scattering theory, Kaizuka \cite{KK} settled the Strichartz conjecture on  Riemannian symmetric spaces of noncompact type, of arbitrary rank.\\
We now describe the contents of this paper. The proofs of  our results are a generalisation of  Kaizuka's method \cite{KK}. In section 2 we recall some basic facts on the quaternionc hyperbolic spaces and introduce the vector Poisson transforms. In section 3, we define the Helgason-Fourier transform on the vector bundles $G\times_K V_\nu$ and give the inversion and Plancherel Theorem. 
The proof of Theorem \ref{spectre} follows from the 
Plancherel formula and Theorem \ref{range}. The  main ingredients in proving Theorem \ref{range} are a Fourier restriction estimate for the vector valued Helgason-Fourier transform (Proposition \ref{rest Fourier} in section 4) and an asymptotic formula for the vector Poisson transform in the framework of Agmon-H\"ormander spaces \cite{AH} (Theorem \ref{asympt}). The proof of Theorem \ref{asympt} will be derived from the Key lemma of this paper giving the asymptotic behaviour of the translate of the $\tau_\nu$-spherical functions.  Section 6 is devoted to the proof of  our main results. In section 7 we prove the Key Lemma. 
\section{Preliminaries}
\subsection{The quaternionic hyperbolic space} 
Let \(G=Sp(n,1)\) be  the group of all linear transformations of the right $\mathbb{H}$-vector space  $\mathbb{H}^{n+1}$ which preserve the  quadratic  form $\displaystyle\sum_{j=1}^n \mid u_j\mid^2-\mid u_{n+1}\mid^2$. 
Let $K=Sp(n)\times Sp(1)$ be the subgroup of $G$ consisting of pairs $(a,d)$ of unitaries. Then  $K$ is  a maximal compact subgroup of $G$. The quaternionic hyperbolic space is the rank one symmetric space \(G/K\) of the noncompact type. It can be realized as the unit ball $\mathbb{B}(\mathbb{H}^n)=\{x\in \mathbb{H}^n; \mid x\mid<1\}$. The group $G$ acts on $\mathbb{B}(\mathbb{H}^n)$ by the fractional linear mappings $x\mapsto g.x=(ax+b)(cx+d)^{-1}$, if $g=\begin{pmatrix}
a&b\\c&d
\end{pmatrix}$, with $a\in \mathbb{H}^{n\times n}, b\in \mathbb{H}^{n\times 1}, c\in \mathbb{H}^{1\times n}$ and $d\in \mathbb{H}$.\\ 
Denote by $\mathfrak{g}$ the Lie algebra of $G$; $\mathfrak{g}=\mathfrak{k}\oplus \mathfrak{p}$ the Cartan decomposition of $\mathfrak{g}$, where $\mathfrak{p}$ is a vector space of matrices of the form
$
\displaystyle\left\lbrace\begin{pmatrix}
0&x\\
x^\ast&0\\
\end{pmatrix}, x\in \mathbb{H}^{n}\right\rbrace
$,
and 
$\displaystyle
\mathfrak{k}=\left\lbrace\begin{pmatrix}
X&0\\
0&q\\
\end{pmatrix}, X^\ast+X=0, q+\overline{q}=0\right\rbrace$,
where $X^\ast$ is the conjugate transpose of the matrix $X$ and $q\in \mathbb{H}$.\\
Let   $H = \begin{pmatrix}
0_n&e_1\\
^t e_1&0\\
\end{pmatrix}\in \mathfrak{p}$ with $^t e_1=(1,0,\cdots,0)$.
Then  $\textbf{a}=\R\,H$ is a Cartan subspace in $\mathfrak{p}$, and the corresponding analytic subgroup  $A=\{ a_t=\exp t\,H; t\in \mathbb{R}\}$, where 
$
\displaystyle a_t= \begin{pmatrix}
cht&0&sht\\
0&0_{n-1}&0\\
sht&0&cht\\
\end{pmatrix}.
$
With $A$ determined we then have that\\
$$
M=\left\lbrace g=\begin{pmatrix}
q&0&0\\
0&m&0\\
0&0&q
\end{pmatrix}, m\in Sp(n-1), \mid q\mid=1\right\rbrace
\simeq Sp(n-1)\times Sp(1).
$$
Let $\alpha\in\mathfrak{a}^\ast$ be defined by $\alpha(H)=1$. Then a system $\Sigma$ of restricted roots  of the pair $(\mathfrak{g},\mathfrak{a})$ is $\Sigma=\{\pm \alpha, \pm 2\alpha\}$ if $n\geq 2$ and $\Sigma=\{ \pm 2\alpha\}$ if $n=1$, with  Weyl group   $W\simeq \{\pm Id\}$. A positive subsystem  of roots corresponding to the positive Weyl chamber $\mathfrak{a}^+\simeq (0,\infty)$ in $\mathfrak{a}$ is $\Sigma^+=\{\alpha, 2\alpha\}$ if $n\geq 2$ and $\Sigma^+=\{ 2\alpha\}$ if $n=1$.\\
 Let $\mathfrak{n}=\mathfrak{g}_\alpha+ \mathfrak{g}_{2\alpha}$ be the direct sum of the positive root subspaces, with $\dim \mathfrak{g}_\alpha=4(n-1)$ and $\dim\mathfrak{g}_{2\alpha}=3$ and $N$ the corresponding analytic subgroup of $G$. Then 
the half sum of the positive restricted roots with multiplicities counted $\rho$ equals to $(2n+1)\alpha$,  and  shall be viewed as a real number $\rho=2n+1$ by the identification $\displaystyle\mathfrak{a}_c^\ast\simeq \mathbb{C}$ via $\displaystyle\lambda\alpha\leftrightarrow\lambda$.\\
Let $\displaystyle \overline{A^+}= \{a_t\in A; \quad t\geq 0\}$. Then we have the Cartan decomposition $G=K\overline{A^+}K$, that is any $g\in G$ can be written  $g=k_1(g)\,{\rm e}^{A^+(g)}\,k_2(g), \quad k_1(g), k_2(g)\in K$ and $A^+(g)\in \overline{\mathfrak{a}^+}$.\\ 
If we write $g\in G$ in $(n+1)\times (n+ 1)$ block notation as
$\displaystyle g=\begin{pmatrix}
a&b\\c&d
\end{pmatrix}$. Then a straightforward computation gives
\begin{align}\label{component}
\cosh A^+(g)= \mid d\mid \quad \textit{and}\quad  H(g)=\log \mid ce_1+d\mid.
\end{align}
We  normalize the invariant measure ${\rm d}g_K$ on  \(G/K\) so that the following integral  formula holds: for all $h\in L^1(G/K)$, 
\begin{align}\label{integ} 
\int_{G/K}h(gK){\rm d}g_K=\int_G h(g.0){\rm d}g= \int_{K}\int_{0}^\infty h(k\,a_t)\Delta(t)\,{\rm d}k\,{\rm d}t, 
\end{align}
where  ${\rm d}t$ is the Lebesgue measure,
$\Delta(t)=(2\sinh t)^{4n-1}(2\cosh t)^3$, and \({\rm d}k\) is the Haar measure of  \(K\) with \(\displaystyle\int_K {\rm d}k=1\).
\subsection{The vector Poisson transform}
In this subsection we define the Poisson transform associated to the vector bundles $G\times_K V_\nu$ over $Sp(n,1)/Sp(n)\times Sp(1)$ and derive some results referring to \cite{Ol}, \cite{V},  and \cite{Y}   for more informations on the subject.\\
Let $\sigma_\nu$ denote the restriction of $\tau_\nu$ to $M$. For $\lambda\in \mathbb{C}$ we consider the representation $\sigma_{\nu,\lambda}$ of $P=MAN$ on $V_\nu$ defined by  $\sigma_{\nu,\lambda}(man)=a^{\rho-i\lambda}\sigma_\nu(m)$. Then $\sigma_{\nu,\lambda}$ defines a principal series representations of $G$ on the Hilbert space 
\begin{align*}
H^{\nu,\lambda}:=\{f:G\rightarrow V_\nu\mid f(gman)=\sigma_{\nu,\lambda}^{-1}(man)f(g)\, \forall man \in MAN, f_{\mid K}\in L^2\},
\end{align*}
where  $G$ acts  by the left regular representation.
We shall denote by $C^{-\omega}(G,\sigma_{\nu,\lambda})$ the  space of its hyperfunctions vectors. By the Iwasawa decomposition, the restriction map from $G$ to $K$ gives an isomorphism from   $H^{\nu,\lambda}$ onto the space $L^2(K,\sigma_\nu)$. This yields, the so-called compact picture of $H^{\nu,\lambda}$, with the group action given by 
\begin{align*}
\pi_{\sigma_\nu,\lambda}(g)f(k)={\rm e}^{(i\lambda-\rho)H(g^{-1}k)}f(\kappa(g^{-1}k)).
\end{align*}
By $C^{-\omega}(K,\sigma_\nu)$ we denote the space of its hyperfunctions vectors.\\
A Poisson transform is the continuous, linear, $G$-equivariant map $\mathcal{P}_\lambda^\nu$ from  $C^{-\omega}(G, \sigma_{\nu,\lambda})$ to $C^\infty(G,\tau_\nu)$ defined by
\begin{align*}
\mathcal{P}^\nu_\lambda\,f(g)=\int_K \tau_\nu(k)f(gk)\,{\rm d}k.
\end{align*}
In the compact picture the Poisson transform is given by 
\begin{align*}
\mathcal{P}^\nu_\lambda\,f(g)=\int_K {\rm e}^{-(i\lambda+\rho)H(g^{-1}k)}\tau_\nu(\kappa(g^{-1}k))\,f(k)\,{\rm d}k.
\end{align*}
Let $\mathbb{D}(G,\tau_\nu)$ denote the algebra of left invariant differential operators on $C^\infty(G,\tau_\nu)$. Let $\mathcal{E}_{\nu,\lambda}(G)$ be the space of all $F\in C^\infty(G,\tau_\nu)$ such that 
$\displaystyle \Omega\, F=-(\lambda^2+\rho^2-\nu(\nu+2))\,F$. 
\begin{pro}
(i) $\mathbb{D}(G,\tau_\nu)$  is the algebra generated by the Casimir operator $\Omega$ of $\mathfrak{g}$.\\
(ii) For $\lambda\in \mathbb{C}, \nu\in \mathbb{N}$, the Poisson transform $\mathcal{P}^\nu_\lambda$ maps $C^{-\omega}(G, \sigma_{\nu,\lambda})$ to $\mathcal{E}_{\nu,\lambda}(G)$.
\end{pro}
\begin{proof}
(i) Let $U(\mathfrak{a})$ be the universal enveloping  algebra of the complexification of $\mathfrak{a}$.  Since the restriction of $\tau_\nu$ to $M$ is irreducible, then $\mathbb{D}(G,\tau_\nu)\simeq U(\mathfrak{a})^W$. As $\mathfrak{a}$ is one dimensional, then $\mathbb{D}(G,\tau_\nu)\simeq \mathbb{C}[s^2]$, symmetric functions of one variable . Thus $\mathbb{D}(G,\tau_\nu)$ is generated by the Casimir element $\Omega$ of the Lie algebra $\mathfrak{g}$ of $G$, viewed as a differential operator acting on $C^\infty(G,\tau_\nu)$.\\
(ii) Since $\sigma_\nu$ is irreducible, the image of $\mathcal{P}^\nu_\lambda$ consists of joint eigenfunctions with respect to the action of $\Omega$. Moreover  $\Omega$ acts by the infinitesimal character of the the principal series representations $\pi_{\sigma_\nu,\lambda}$. It follows from  Proposition 8.22 and Lemma 12.28 in \cite{Kn}, that 
\begin{align}\label{eigenvalue}
\pi_{\sigma_\nu,\lambda}(\Omega)=-(\lambda^2+\rho^2-c(\sigma_\nu))Id \quad \textit{on}\quad C^{-\omega}(G, \sigma_{\nu,\lambda}),
\end{align}
where $c(\sigma_\nu)$ is the Casimir value of $\sigma_\nu$ given by $c(\sigma_\nu)=\nu(\nu+2)$.\\
\end{proof}
Let $\Phi_{\nu,\lambda}$ be the $\tau_\nu$-spherical function associated to $\sigma_\nu$. Then $\Phi_{\nu,\lambda}$   admits  the following Eisenstein integral representation (see [\cite{com}, Lemma 3.2]):
\begin{align*}
\Phi_{\nu,\lambda}(g)=\int_K {\rm e}^{-(i\lambda+\rho)H(g^{-1}k)}\tau_\nu(\kappa(g^{-1}k)k^{-1})\, {\rm d}k.
\end{align*}
Note that $\Phi_{\nu,\lambda}$ lies in  
$
C^\infty(G,\tau_\nu,\tau_\nu)$ the space of smooth functions $F:G\rightarrow End (V_{\tau_\nu})$ satisfying 
\begin{align*}
F(k_1 g k_2)=\tau_\nu(k_2^{-1})F(g)\tau_\nu(k_1^{-1}),
\end{align*}
the so called  $\tau_\nu$-radial functions. Being $\tau_\nu$-radial, $\Phi_{\nu,\lambda}$ is completely determined by its restriction to $A$, by the Cartan decomposition $G=KAK$. Moreover, since $\sigma_\nu$ is irreducible, it follows that $\Phi_{\nu,\lambda}(a_t)\in End_M(V_\nu)\simeq \mathbb{C}Id_{V_\nu}$, $\forall a_t\in A$. Therefore  there exists $\varphi_\nu:\mathbb{R}\rightarrow \mathbb{C}$ such that $\Phi_{\nu,\lambda}(a_t)=\varphi_\nu(t).Id_{V_\nu}$. We have
\begin{align}
\varphi_{\nu,\lambda}(t)=\frac{1}{\nu+1}\int_K {\rm e}^{-(i\lambda+\rho)H(g^{-1}k)}\chi_\nu(\kappa(g^{-1}k)k^{-1})\, {\rm d}k,
\end{align}
where $\chi_\nu$ is the character of $\tau_\nu$.\\
This so-called trace $\tau_\nu$-spherical function  has  been computed  explicitly in \cite{VP} using the radial part of the Casimir operator $\Omega$ (see also \cite{T} ).
We have $\varphi_{\nu,\lambda}(t)=(\cosh t)^\nu \phi_\lambda^{(\rho-2,\nu+1)}(t)$,
where $\phi_\lambda^{(\rho-2,\nu+1)}(t)$ is the Jacobi function (cf. \cite{Ko})
\begin{align*}
\phi_\lambda^{(\rho-2,\nu+1)}(t)=\,_2 F_1(\frac{i\lambda+\rho+\nu}{2},\frac{-i\lambda+\rho+\nu}{2}; \rho-1; -\sinh^2 t).
\end{align*}
We deduce from \eqref{A4} the asymptotic behaviour of $\varphi_{\nu,\lambda}$ 
\begin{align}\label{asymp}
\varphi_{\lambda,\nu}(a_t)={\rm e}^{(i\lambda-\rho)t}[\mathbf{c}_\nu(\lambda)+\circ(1)], \,\, \textit{as}\,\, \,t\rightarrow \infty\quad  if \quad \Im(\lambda)<0.
\end{align}
where
\begin{align}\label{c}
\mathbf{c}_\nu(\lambda)=\frac{2^{\rho-i\lambda}\Gamma(\rho-1)\Gamma(i\lambda)}{\Gamma(\frac{i\lambda+\rho+\nu}{2})\Gamma(\frac{i\lambda+\rho-\nu-2}{2})}.
\end{align}
For $\lambda\in \mathbb{C}$ the $\mathbf{c}$-function of Harish-Chandra associated to $\tau_\nu$  is defined by
\begin{align*}
\mathbf{c}(\tau_\nu,\lambda)=\int_{\overline{N}}{\rm e}^{-(i\lambda+\rho)H(\overline{n})}\tau_\nu(\kappa(\overline{n}))\,{\rm d}\overline{n}.
\end{align*}
The integral converges for $\lambda$ such that $\Re(i\lambda)>0$ and it has a meromorphic continuation to $\mathbb{C}$.\\ In above ${\rm d}\overline{n}$ is the Haar measure of $\overline{N}=\theta(N)$, $\theta$ being the Cartan involution.\\
We may use formula \eqref{c} to give explicitly $\mathbf{c}(\tau_\nu,\lambda)$. Indeed, one easily check that $\mathbf{c}(\tau_\nu,\lambda)\in End_M(V_\nu)= \mathbb{C}Id_{V_\nu}$.
Then using the following result on the behaviour of $\Phi_{\nu,\lambda}(a_t)$ (\cite{Y}, Proposition 2.4)
\begin{align*}
\Phi_{\nu,\lambda}(a_t)={\rm e}^{(i\lambda-\rho)t}(\mathbf{c}(\tau_\nu,\lambda)+\circ(1)) \textit{as}\quad t\rightarrow \infty,
\end{align*}
together with  $\Phi_{\nu,\lambda}(a_t)=\varphi_{\nu,\lambda}(t). Id$, we  find then from  \eqref{asymp} that  $\mathbf{c}(\tau_\nu,\lambda)=\mathbf{c}_\nu(\lambda)Id_{V_\nu}$.\\
We end this section by recalling a result of Olbrich \cite{Ol} on the range of the Poisson transform on vector bundles which reads in our case as follows
\begin{theo}\label{O}\cite{Ol}
Let $\nu\in \mathbb{N}$ and  $\lambda\in \mathbb{C}$ such that 
\begin{enumerate}
\item[(i)] $-2i\lambda\notin \mathbb{N}$
\item[(ii)] $i\lambda+\rho\notin -2\mathbb{N}-\nu\cup -2\mathbb{N}+\nu+2$.
\end{enumerate}
Then the Poisson transform $\mathcal{P}^\nu_\lambda$ is a $K$-isomorphism from $C^{-\omega}(K,\sigma_\nu)$ onto $\mathcal{E}_{\nu,\lambda}(G)$.
\end{theo}
\section{The vector-valued Helgason-Fourier transfrorm}
In this section we give the inversion  and the Plancherel formulas for the Helgason-Fourier transform on the vector bundle $G\times_K V_\nu$.\\ 
According to \cite{com} the  vector-valued Helgason-Fourier transform of \(f\in C^{\infty}_c(G,\tau_\nu)\)  is the $V_\nu$-valued function on $\mathbb{C}\times K$   defined  by:
\[\mathcal{F}_\nu f(\lambda,k)=\int_G e_{\lambda,\nu}(k^{-1}g)\,f(g){\rm d}g ,\]
where $e_{\lambda,\nu}$ is the vector valued function $e_{\lambda,\nu}: G\rightarrow End(V_\nu)$ given by 
\begin{align*}
e_{\lambda,\nu}(g)= {\rm e}^{(i\lambda-\rho)H(g^{-1})}\tau_{\nu}^{-1}(\kappa(g^{-1})).
\end{align*}
Notice that our sign on "$\lambda$" is the opposite of the one  in \cite{com}.\\
In order to state the next theorem, we  introduce the finite set in $\{\lambda, \Im(\lambda)\geq 0\}$
\begin{align*}
D_\nu=\{\lambda_j=i(\nu-\rho+2-2j), j=0,1,\cdots,\nu-\rho+2-2j>0\}.
\end{align*}
Note that  $D_\nu$ is empty if $\nu\leq \rho-2$. It parametrizes the discrete series representation of $G$ containing $\tau_\nu$, see \cite{VP}.\\ Let
\begin{align*}
d_\nu(\lambda_j)=\frac{2^{-2(\rho-\nu-1)}(\nu-\rho-2j+2)(\rho-2+j)!(\nu-j)!}{\Gamma^2(\rho-1)j!(\nu-\rho-j+2)!}, \quad \lambda_j\in D_\nu
\end{align*} 
For $\lambda_j\in D_\nu$, we define  the operators $\mathcal{Q}^\nu_j$ 
\begin{align*}
\begin{array}{ll}
L^2(G,\tau_\nu) \rightarrow  \mathcal{E}_{\nu,\lambda_j}(G,\tau_\nu)\\
F \mapsto d_\nu(\lambda_j)\,\Phi_{\nu,\lambda_j}\ast F
\end{array}
\end{align*}
We denote the image by $A^2_j$. We set 
$$
L^2_{\textit{disc}}(G,\tau_\nu)=\bigoplus_{j;\, \nu-\rho+2-2j>0} A^2_j,
$$
and denote by  $L^2_{\textit{cont}}(G,\tau_\nu)$  its orthocomplement.
Let  $L^2_{\sigma_\nu}(\mathbb{R}^+\times K, \mid \mathbf{c}_\nu(\lambda)\mid^{-2}{\rm d}\lambda\, {\rm d}k)$ be the space of vector functions $\phi:\mathbb{R}^+\times K\rightarrow V_\nu$ satisfying 
\begin{enumerate}
\item[(i)] For each fixed $\lambda, \phi(\lambda,km)=\sigma_\nu(m)^{-1}\phi(\lambda,k), \forall m\in M$
\item[(ii)] $\int_{\mathbb{R}^+\times K}\parallel \phi(\lambda,k)\parallel_\nu^2\, \mid \mathbf{c}_\nu(\lambda)\mid^{-2}{\rm d}\lambda\,{\rm d}k<\infty$.
\end{enumerate}
\begin{theo}\label{fourier}
(i) For  $F\in C_c^\infty(G,\tau_\nu)$ we have the following inversion and Plancherel formulas
\begin{equation}\label{inversion1}
 F(g)=\frac{1}{2\pi}\int_0^{\infty}\int_K e_{\lambda,\nu}^\ast(k^{-1}g)\mathcal{F}_\nu F(\lambda,k)\,\mid \mathbf{c}_\nu(\lambda)\mid^{-2}{\rm d}\lambda\,{\rm d}k
+\sum_{\lambda_j\in D_\nu} d_\nu(\lambda_j)\int_K e_{\lambda_j,\nu}^\ast(k^{-1}g)\mathcal{F}_\nu F(\lambda_j,k)\,{\rm d}k,
\end{equation}
\begin{equation}\label{plancherel1}
\int_G \parallel F(g)\parallel_\nu^2 {\rm d}g_K=\frac{1}{2\pi}\int_0^{\infty}\int_K \parallel \mathcal{F}_\nu F(\lambda,k)\parallel_\nu^2\mid \mathbf{c}_\nu(\lambda)\mid^{-2}{\rm d}\lambda\,{\rm d}k 
+\sum_{\lambda_j\in D_\nu} d_\nu(\lambda_j)\int_K <\mathcal{F}_\nu F(\lambda_j,k), \mathcal{F}_\nu F(-\lambda_j,k)>_\nu \,{\rm d}k
\end{equation}
(ii) The Fourier transform $\mathcal{F}_\nu$ extends to an isometry from $L^2_{\textit{cont}}(G,\tau_\nu)$ onto the space $L_{\sigma_\nu}^2(\mathbb{R}^+\times K,\mid \mathbf{c}_\nu(\lambda)\mid^{-2}{\rm d}\lambda\, {\rm d}k)$.
\end{theo}
The first part of Theorem \ref{fourier} can  be easily deduced from the inversion and Plancherel formulas for the spherical transform.\\ 
Let $C^\infty_c(G,\tau_\nu,\tau_\nu)$ denote the space of smooth compactly supported $\tau_\nu$-radial functions. The spherical transform  of $F\in C^\infty_c(G,\tau_\nu,\tau_\nu)$ is the \({\C}\)-valued function \(\mathcal{H}_\nu F\) defined by:
$$
\mathcal{H}_\nu F(\lambda)=\frac{1}{\nu+1}\int_G Tr [\Phi_{\nu,\lambda}(g^{-1})F(g))]{\rm d}g, \quad \lambda\in \mathbb{C}.
$$ 
The inversion and the Plancherel formulas for the $\tau$-spherical transform have been given explicitly in \cite{VP}. For the convenience of the reader we give an elementary  proof by using the Jacobi transform.
\begin{theo}\label{sphi}
For $F\in C_c^\infty(G,\tau_\nu,\tau_\nu)$ we have the following inversion and Plancherel formulas
\begin{equation}\begin{split}\label{inversion2}
F(g)=\frac{1}{2\pi }\int_0^{+\infty}\Phi_{\nu,\lambda}(g)\mathcal{H}_\nu F(\lambda)\mid \mathbf{c}_\nu(\lambda)\mid^{-2}{\rm d}\lambda
+\sum_{\lambda_j\in D_\nu}\Phi_{\nu, \lambda_j}(g)\mathcal{H}_\nu f(\lambda_j)\, d_\nu(\lambda_j),
\end{split}\end{equation}
\begin{equation}\begin{split}\label{plancherel2}
 \int_G \parallel F(g)\parallel_{HS}^2 {\rm d}g=\frac{\nu+1}{2\pi }\int_0^{+\infty}\mid\mathcal{H}_\nu F((\lambda)\mid^2\mid \mathbf{c}_\nu(\lambda)\mid^{-2}{\rm d}\lambda+(\nu+1)\sum_{\lambda_j\in D_\nu} d_\nu(\lambda_j)\mid \mathcal{H}_\nu F((\lambda_j)\mid^2,
\end{split}\end{equation}
\end{theo}
In above $\parallel\,\parallel_{HS}$ stands for the Hilbert-Schmidt norm.
\begin{proof}
Let $F\in C_c^\infty(G,\tau_\nu,\tau_\nu)$ and let $f_\nu$ be its scalar component. Using  the integral formula \eqref{integ}, the identity $\Phi_{\nu,\lambda}(a_t)=\Phi_{\nu,\lambda}(a_{-t})=(\cosh t)^\nu\phi_\lambda^{(\rho-2,\nu+1)}(t)$ and the fact that $  \Delta(t)=(2\cosh t)^{-2\nu}\Delta_{\rho-2,\nu+1}$, we have
\begin{align}\begin{split}\label{sphi0}
\mathcal{H}_\nu F(\lambda)&=\int_0^\infty f_\nu(t)(\cosh t)^\nu\phi_\lambda^{(\rho-2,\nu+1)}(t)\,\Delta(t)\,{\rm d}t\\
&=\int_0^\infty f_\nu(t)(2^2\cosh t)^{-\nu}\phi_\lambda^{(\rho-2,\nu+1)}(t)\,\Delta_{\rho-2,\nu+1}(t)\,{\rm d}t.
\end{split}\end{align}
Thus the $\tau_\nu$-spherical transform $\mathcal{H}_\nu F$ may be written in terms of the Jacobi transform $\mathcal{J}^{\alpha,\beta}$, with $\alpha=\rho-2$ and $\beta=\nu+1$. Namely, we have
$$
\mathcal{H}_\nu F(\lambda)=\mathcal{J}^{\rho-2,\nu+1}[(2^2 \cosh t)^{-\nu}f_\nu](\lambda).
$$
We refer to \eqref{A5} in the Appendix for the definition of the Jacobi transform.\\
Now the theorem follows from the inversion and the Plancherel formulas for  the Jacobi transform \eqref{A6}, \eqref{A6'} and \eqref{A7} in the Appendix.
\end{proof}
For the proof of the surjectivity statement in Theorem \ref{fourier} we shall  need the following result
\begin{pro}\label{FS1}
Let $F\in C^{\infty}_c(G,\tau_\nu)$ and $\Phi\in C^\infty(G,\tau_\nu,\tau_\nu)$. Then we have
\begin{align*}
\mathcal{F}_\nu (F\ast \Phi)(\lambda,k)=\mathcal{H}_\nu \Phi(\lambda)\mathcal{F}_\nu F(\lambda,k), \quad \lambda\in \mathbb{C}, k\in K,
\end{align*}
where the convolution is defined by
\begin{align*}
(\Phi\ast F)(g)=\int_G \Phi_{\nu,\lambda}(x^{-1}g)F(x)\,{\rm d}x.
\end{align*}
\end{pro}
\begin{proof}
Let $\Phi\in C^\infty(G,\tau_\nu,\tau_\nu)$, $v\in V_\nu$, and  set $F_v=\Phi(.\,)v$. Then we have the following relation between the Fourier transform and the spherical transform
\begin{align}\label{FS}
\mathcal{F}_\nu F_v(\lambda,k)=\mathcal{H}_\nu \Phi(\lambda)\tau(k^{-1})v.
\end{align}
By definition 
\begin{align*}\begin{split}
\mathcal{F}_\nu(F\ast \Phi)(\lambda,k)&=\int_G\,\int_G e_\lambda^\nu(k^{-1}g)\Phi(x^{-1}g)F(x){\rm d}x{\rm d}g\\
&=\int_G {\rm d}x\int_G e_\lambda^\nu(k^{-1}xy)\Phi(y)F(x){\rm d}y
\end{split}\end{align*}
Using the following cocycle relations for the Iwasawa function $H(x)$ 
\begin{align*}
 H(xy)=H(x\kappa(y))+H(y),
\end{align*} 
and
\begin{align*}
\kappa(xy)=\kappa(x\kappa(y)),
\end{align*}
for all $x,y \in G$, we get the following identity 
\begin{align*}
e_\lambda^\nu(k^{-1}xy)={\rm e}^{(i\lambda-\rho)H(x^{-1}k)}e_\lambda^\nu(\kappa^{-1}(x^{-1}k)y),
\end{align*}
from which we obtain
\begin{align*}
\mathcal{F}_\nu(\Phi\ast F)(\lambda,k)=\int_G {\rm e}^{(i\lambda-\rho)H(x^{-1}k)}\left( \int_G e_{\lambda,\nu}(\kappa^{-1}(x^{-1}k) y)\Phi(y)F(x)\,{\rm d}y\right){\rm d}x.
\end{align*}
Next, put $\displaystyle h_v(y)=\Phi(y)v, v\in V_{\tau_\nu}$. Then  \eqref{FS} implies
\begin{align*}
\int_G e_{\lambda,\nu}(\kappa^{-1}(x^{-1}k) y)\Phi(y)F(x)\,{\rm d}y &=\mathcal{F}_\nu (h_{F(x)})(\lambda,\kappa^{-1}(x^{-1}k))\\
&=\mathcal{H}(\Phi)(\lambda)\tau_\nu(\kappa^{-1}(x^{-1}k))F(x),
\end{align*}
from which we deduce
\begin{align*}
\mathcal{F}_\nu(\Phi\ast F)(\lambda,k)=\mathcal{H}(\Phi)(\lambda)\int_G {\rm e}^{(i\lambda-\rho)H(x^{-1}k)}\tau_\nu(\kappa^{-1}(x^{-1}k))F(x){\rm d}x,
\end{align*}
and the proposition follows.
\end{proof}
We now come to the proof of  Theorem \ref{fourier}.
\begin{proof}
(i) We may follow the same method as in \cite{com}  to prove the  inversion formula \eqref{inversion1} and the Plancherel formula \eqref{plancherel1}  from Theorem \ref{sphi}. We give an outline of the proof.\\ Let $F\in C^\infty_c(G,\tau_\nu)$ and consider the $\tau_\nu$-radial function defined for any $g\in G$ by $$
\displaystyle F_{g,v}(x).w=\int_K<\tau_\nu(k)w,v>_\nu F(gkx)\, {\rm d}k,
$$
$v$ being a fixed vector in $ V_\nu$. Then a straightforward calculation shows that
\begin{align*}
\mathcal{H}_\nu F_{g,v}(\lambda)=\frac{1}{\nu+1}<(\Phi_{\nu,\lambda}\ast F)(g),v>_\nu.
\end{align*}
The inversion formula for the spherical transform together with  $Tr F_{g,v}(e)=<F(g),v>_\nu$ imply
\begin{align*}
F(g)=\frac{1}{2\pi}\int_0^ \infty (\Phi_{\nu,\lambda}\ast F)(g)\mid \mathbf{c}_\nu(\lambda)\mid^{-2}\, {\rm d}\lambda+\sum_{\lambda_j\in D_\nu}(\Phi_{\nu,\lambda_j}\ast F)(g)d_\nu(\lambda_j).
\end{align*}
To conclude use the following result for the translated spherical function ( see \cite{com} Proposition 3.3)
\begin{align}\label{symm}
\Phi_{\nu,\lambda}(x^{-1}y)=\int_K {\rm e}^{-(i\lambda+\rho)H(y^{-1}k)}{\rm e}^{(i\lambda-rho)H(x^{-1}k)}\tau_\nu(\kappa(y^{-1}k))\tau_\nu(\kappa^{-1}(x^{-1}k))\, {\rm d}k,
\end{align} 
to get 
\begin{align*}
(\Phi_{\nu,\lambda}\ast F)(g)=\int_K {\rm e}^{-(i\lambda+\rho)H(g^{-1}k)} \tau_\nu(\kappa(g^{-1}k))\mathcal{F}_\nu F(\lambda,k)\, {\rm d}k,
\end{align*}
and the inversion formula \eqref{inversion1} follows.\\
The proof of the Plancherel formula \eqref{plancherel1} is essentially the same as in the scalar case, so we omit it.\\
Note that  as a consequence of the Plancherel formula not involving the discrete series, we have
\begin{align*}
\int_G\parallel F(g)\parallel_\nu^2\, {\rm d}g_K=\frac{1}{\pi}\int_0^\infty\int_K \parallel\mathcal{F}_\nu F(\lambda,k)\parallel_\nu^2\,\mid \mathbf{c}_\nu(\lambda)\mid^{-2}{\rm d}\lambda\,{\rm d}k,
\end{align*}
for every $F\in L_{\textit{cont}}^2(G,\tau_\nu)$.\\
(ii) We prove the surjectivity statement. Suppose that there exists a function $f$ in $L^2_{\sigma_\nu}(\mathbb{R}^+\times K,\mid \mathbf{c}_\nu(\lambda)\mid^{-2}{\rm d}\lambda\, {\rm d}k)$ such that 
\begin{align*}
\int_0^\infty\int_K<f(\lambda,k),\mathcal{F}_\nu F(\lambda,k)>\mid \mathbf{c}_\nu(\lambda)\mid^{-2}{\rm d}\lambda \,{\rm d}k=0
\end{align*}
for all $F\in C_c^\infty(G,\tau_\nu)$. Changing $F$ into $F\ast\Phi$ where $\Phi\in C^\infty(G,\tau_\nu,\tau_\nu)$ and using Proposition \ref{FS1}, we have
\begin{align*}
\int_0^\infty\int_K <f(\lambda,k),\mathcal{F}_\nu F(\lambda,k)>\, \mathcal{H}_\nu \phi(\lambda)\mid \mathbf{c}_\nu(\lambda)\mid^{-2}{\rm d}\lambda \,{\rm d}k=0
\end{align*}
By the Stone-Weierstrass theorem, the algebra $\{\mathcal{H}_\nu \Phi, \Phi\in C^\infty(G,\tau_\nu,\tau_\nu)\}$ is dense in $C_e^\infty(\mathbb{R})$ the space of even continuous functions on $\mathbb{R}$ vanishing at infinity. Therefore for every $F\in C_c^\infty(G,\tau_\nu)$ there is a set $E_F$ of measure zero in $\mathbb{R}$ such that 
\begin{align*}
\int_K <f(\lambda,k),\mathcal{F}_\nu F(\lambda,k)> {\rm d}k=0
\end{align*}
for all $\lambda$ not in $E_F$ . The rest of the proof is based on an adaptation of the arguments given in \cite{H} Theorem 1.5, for the scalar case, and the proof of Theorem \ref{fourier} is completed.
\end{proof}
\section{Fourier restriction estimate}	
The main result of this section is the following  uniform continuity estimate for the Fourier-Helgason restriction operator. 
\begin{pro}\label{rest Fourier}
Let $\nu\in \mathbb{N}$. There exists a positive constant \(C_\nu\) such that for $\lambda \in \mathbb{R} \backslash \{0\}$ and $R>1$, we have
 \begin{align}\label{rest}
\bigg(\int_{K}\|\mathcal{F}_\nu F(\lambda,k)\|_\nu^{2}dk\bigg)^{1/2} \leq C_{\nu} |c_\nu(\lambda)| R^{1/2} \bigg(\int_{G/K}\|F(g)\|_\nu^{2}\, {\rm d}g_K\bigg)^{1/2},
\end{align}
for every $F\in L^{2}(G,\tau_\nu)$ with  \(\textit{supp} F \subset B(R)\).
\end{pro}
To prove this result we shall need estimates of the Harish-Chandra $c$-function. To this end we introduce  the function $\mathbf{b}_\nu(\lambda)$ defined on $\mathbb{R}$ by 
\begin{align*}
\mathbf{b}_\nu(\lambda)=
\begin{cases}
\mathbf{c}_\nu(\lambda)\quad \textit{if} \quad \frac{\nu-\rho+2}{2}\in \mathbb{Z}^+\\
\lambda\,\mathbf{c}_\nu(\lambda)\quad \textit{if}\quad  \frac{\nu-\rho+2}{2}\notin \mathbb{Z}^+
\end{cases}\,
\end{align*}
\begin{lem}
Assume $\nu>\rho-2$.
\begin{enumerate}
\item[(i)] The function $\mathbf{b}_\nu(\lambda)$ has no zero in $\mathbb{R}$.
\item[(ii)] There exists a positive constant $C$ such that for $\lambda\in \mathbb{R}$, we have
\begin{align}\label{b}
C^{-1}(1+\lambda^2)^{\frac{2\rho-4-\varepsilon(\nu)}{4}}\leq \mid \mathbf{b}_\nu(\lambda)\mid^{-1}\leq C (1+\lambda^2)^{\frac{2\rho-4-\varepsilon(\nu)}{4}},
\end{align}
\end{enumerate}
with $\varepsilon(\nu)=\pm 1$ according to $\frac{\nu-\rho+2}{2}\notin \mathbb{Z}^+$ or $\frac{\nu-\rho+2}{2}\in \mathbb{Z}^+$
\end{lem}
\begin{proof}
\begin{enumerate}
\item[(i)] If $\frac{\nu-\rho+2}{2}\notin \mathbb{Z}^+$, then $\mathbf{b}_\nu(\lambda)=\frac{2^{\rho+\nu-i\lambda}\Gamma(\rho-1)\Gamma(i\lambda+1)}{\Gamma(\frac{i\lambda+\rho+\nu}{2})\Gamma(\frac{i\lambda+\rho-\nu-2}{2})}$, and clearly $\mathbf{b}_\nu(\lambda)$ has no zero on $\mathbb{R}$.\\
If $\frac{\nu-\rho+2}{2}\in \mathbb{Z}^+$ then $\mathbf{b}_\nu(\lambda)$ a priori can have zero and pole at $\lambda=0$.  This is not the case, since 
\begin{align*}
\lim_{\lambda\rightarrow 0}\mathbf{b}_\nu(\lambda)=(-1)^{\frac{\nu-\rho+2}{2}}\frac{2^{\rho+\nu}\Gamma(\rho-1)(\frac{\nu-\rho+2}{2})!}{\Gamma(\frac{\rho+\nu}{2})}.
\end{align*}
\item[(ii)] To prove the estimate \eqref{b} we shall use the following property of the $\Gamma$-function
\begin{align}\label{gamma}
\lim_{\mid z\mid\rightarrow \infty}\frac{\Gamma(z+a)}{\Gamma(z)}z^{-a}=1, \, \, \mid \arg(z)\mid<\pi-\delta,
\end{align}
where $a$ is any complex number, and  $\log$ is the principal value of the logarithm and $\delta>0$.\\
Assume first that $\frac{\nu-\rho+2}{2}\notin \mathbb{Z}^+$. Using the duplicata formula for the function gamma
\begin{align*}
\Gamma(2z)=\frac{2^{2z-2}}{\sqrt{\pi}}\Gamma(z)\Gamma(z+\frac{1}{2}),
\end{align*}
we rewrite $ \mathbf{b}_\nu(\lambda)$ as 
\begin{align*}
\mathbf{b}_\nu(\lambda)=\frac{2^{\rho+\nu-1}}{\sqrt{\pi}}\frac{\Gamma(\frac{i\lambda+1}{2})\Gamma(\frac{i\lambda+2}{2})}{\Gamma(\frac{i\lambda+\rho+\nu}{2})\Gamma(\frac{i\lambda+\rho-\nu-2}{2})}.
\end{align*}
It follows from \eqref{gamma} that for every $\lambda\in \mathbb{R}$, we have
\begin{align*}
\mid \mathbf{b}_\nu(\lambda)\mid\leq C(1+\lambda^2)^{-\frac{2\rho-5}{4}}
\end{align*}
and 
\begin{align*}
\mid \mathbf{b}_\nu(\lambda)\mid^{-1}\leq C(1+\lambda^2)^{\frac{2\rho-5}{4}}.
\end{align*}
The proof for the case $\frac{\nu-\rho+2}{2}\in \mathbb{Z}^+$ follows the same line as in the case $\frac{\nu-\rho+2}{2}\notin \mathbb{Z}^+$, so we omit it.
\end{enumerate}
This finishes the proof of the Lemma.
\end{proof}
Let us recall from \cite{A} an auxiliary lemma which will be useful for the proof of Proposition \ref{rest Fourier}.\\
Let \(\eta\) be a positive Schwartz function on \({\R}\) whose Fourier transform has a compact support. For \(m\in {\R}\), set
\begin{equation*}
\eta_{m}(x)=\int_{\R}\eta(t)(1+|t-x|)^{m/2}\, {\rm d}t.
\end{equation*}
\begin{lem}\label{Ank}
\item[i)] \(\eta_m\) is a positive \(C^\infty\)-function with
\begin{equation}
C^{-1}(1+t^2)^{\frac{m}{2}} \leq \eta_{m}(t) \leq C(1+t^2)^{\frac{m}{2}},
\end{equation}
for some positive constant \(C\).
\item[ii)] The Fourier transform of \(\eta_m\) has a compact support.
\end{lem}
In order to prove the Fourier restriction Theorem, we need to introduce the bundle valued Radon transform, see \cite{BOS} for more informations.\\
The Radon transform for \(F\in C_c^\infty(G,\tau_{\nu})\) is defined  by
\[\mathcal{R}F(g)=   e^{\rho H(g)}\int_{N}F(gn) dn.\]
We set $\mathcal{R}F(t,k)=\mathcal{R}F(k a_t)$. Then, using the Iwaswa decomposition $G=NAK$, we may rewrite the Helgason-Fourier transform as
\begin{equation*}
 \mathcal{F}_\nu F(\lambda,k) = \mathcal{F}_{\mathbb{R}}(\mathcal{R}F(\cdot,k ))(\lambda),
\end{equation*}
where   
$$
\mathcal{F}_{\mathbb{R}}\phi(\lambda)=\int_{\mathbb{R}} e^{-i\lambda t}\phi(t)\,{\rm d}t, 
$$   
is the Euclidean Fourier transform of $\phi$ a $V_\nu$-valued smooth function with compact support in $\mathbb{R}$.\\
We define on  $\mathfrak{p}$ the scalar product $<X,Y>=\frac{1}{2}Tr(XY)$ and denote by $\mid \,\,\mid$ the corresponding norm.  It induces  a distance function  $d$ on $G/K$. By the Cartan decomposition $G=K\exp \mathfrak{p}$, any $g\in G$ may be written uniquely as $g=k\exp X$, so that $d(0,gK)=\mid X\mid$. Define the open ball centred at $0$ and of radius $R$ by $B(R)=\{gK\in G/K;\quad d(0,gK)<R\}$.
\begin{lem}\label{Radon}
Let \(F\in C_0^\infty(G,\tau_\nu)\). 
 If $ \textit{supp}\, F\subset \overline{B(R)}$, then $\textit{supp}\, \mathcal{R} F\subset [-R,R]\times K$.
\end{lem}
\begin{proof}
As (see [\cite{H1}, page 476]
$$
d(0, k{\rm e}^{tH}nK)\geq \mid t\mid,\quad  k\in K , n\in N, t\in \mathbb{R}
$$
it follows that $\textit{supp}\, \mathcal{R} F\subset [-R,R]\times K$ if $ \textit{supp}\, F\subset \overline{B(R)}$
\end{proof}
\textbf{Proof of Proposition \ref{rest Fourier}.} 
It suffices to prove the estimate \eqref{rest} for functions $F\in C_c^\infty(G,\tau_\nu)$ supported in $B(R)$. It follows from the Plancherel formula \eqref{plancherel1} that 
\begin{align*}
\int_{B(R)}\parallel F(g)\parallel_\nu^2\,{\rm d}g_K\geq \int_K\int_{\mathbb{R}}\parallel \mathcal{F}_\nu F(\lambda,k)\parallel_\nu^2\,\mid \mathbf{c}_\nu(\lambda)\mid^{-2} {\rm d}\lambda\,{\rm d}k
\end{align*}
Therefore it is sufficient to show
\begin{align}\label{f1}
\int_K\int_{\mathbb{R}}\parallel \mathcal{F}_\nu F(\lambda,k)\parallel_\nu^2\,\mid \mathbf{c}_\nu(\lambda)\mid^{-2} {\rm d}\lambda\,{\rm d}k\geq C\,\frac{\mid \mathbf{c}_\nu(\lambda)\mid^{-2}}{R}\int_{\mathbb{R}}\parallel \mathcal{F}_\nu F(\lambda,k)\parallel_\nu^2\, {\rm d}k,
\end{align}
for some positive constant $C$.\\
By \eqref{b} we have
$
\mid \mathbf{c}_\nu(\lambda)\mid^{-1}\asymp \eta_{\frac{2\rho-3}{2}}(\lambda).
$
Therefore  \eqref{f1} is equivalent to 
\begin{align}\label{f2}
\frac{\eta_{\frac{2\rho-3}{2}}(\lambda)}{R}\int_K \parallel \mathcal{F}_\nu F(\lambda,k)\parallel_\nu^2\, {\rm d}k\leq
\int_K\int_{\mathbb{R}}\parallel \mathcal{F}_\nu F(\lambda,k)\parallel_\nu^2\, \eta_{\frac{2\rho-3}{2}}(\lambda){\rm d}\lambda\,{\rm d}k
\end{align}
Let $T$ be the tempered distribution on $\mathbb{R}$ defined by $T:=\mathcal{F}^{-1}_{\mathbb{R}} \eta_{\frac{2\rho-3}{2}}$.   By Lemma \ref{Ank}, $T$ is compactly supported . Let $R_0>1$ such that  $\mathit{supp}\,T\subset [-R_0,R_0]$. Then \eqref{f2} is equivalent to 
\begin{align}\label{f3}
\int_K\parallel\mathcal{F}_{\mathbb{R}}(T\ast \mathcal{R}F(.\,,k))(\lambda)\parallel_\nu^2\,{\rm d}k\leq C R\int_K\int_{\mathbb{R}}\mathcal{F}_{\mathbb{R}}(T\ast \mathcal{R}F(.\,,k))(\lambda)\parallel_\nu^2\,{\rm d}\lambda\,{\rm d}k,
\end{align}
where $\ast$ denotes the convolution on $\mathbb{R}$.\\
From $\textit{supp} T\subset [-R_0,R_0]$ and Lemma \ref{Radon}, it follows that  for any $k\in K$, $\textit{supp} \,(T\ast \mathcal{R}F(.\,,k)) \subset [-(R+R_0), R+R_0]$.\\
Thus
\begin{align*}
\int_K \parallel\mathcal{F}_{\mathbb{R}}(T\ast\mathcal{R}F(.\,,k)(\lambda)\parallel_\nu^2\, {\rm d}k \leq 2(R+R_0)\int_K\int_{\mathbb{R}}\parallel(T\ast\mathcal{R}F(.\, k))(t)\parallel_\nu^2\,{\rm d}t\, {\rm d}k
\end{align*}
Next use the Euclidean Plancherel formula  to get \eqref{f3}, and the proof is finished.\\
As a  consequence of Proposition \ref{rest Fourier}, we obtain the uniform continuity estimate for the Poisson transform $\mathcal{P}^\nu_\lambda$.
\begin{cor}\label{nc} Let $\nu\in\mathbb{N}$. There exists a positive constant \(C_{\nu}\) such that for $\lambda \in \mathbb{R} \backslash \{0\}$, we have
\begin{align}\label{N1}
\underset{R > 1}{\sup}\left( \frac{1}{R}\int_{B(R)}\parallel\mathcal{P}^\nu_\lambda f(g)\parallel_\nu^{2}{\rm d}g_K\right)^{1/2}\leq C_\nu\,  |c_\nu(\lambda)| \parallel f\parallel_{L^{2}(K,\sigma_\nu)}
\end{align}
for every $f\in L^{2}(K,\sigma_\nu)$.
\end{cor}
\begin{proof} Let $F\in L^2(G,\tau_\nu)$ with $\textit{supp}\,F\subset B(R)$, and let \(f\in L^{2}(K,\sigma_\nu)\). Since \(\lambda\) is real and $\tau_\nu$ is unitary,  the Poisson transform and the restriction Fourier transform are related by the following formula
\begin{align*}
\int_{B(R)}<\mathcal{P}^\nu_\lambda f(g), F(g)>_\nu{\rm d}g_K=\int_K <f(k),\mathcal{F}_\nu F(\lambda,k) >_\nu{\rm d}k.
\end{align*}
Thus 
\begin{align*}\begin{split}
\mid \int_{B(R)}<\mathcal{P}^\nu_\lambda f(g), F(g)>_\nu{\rm d}g_K\mid &\leq \|f\|_{L^{2}(K,\sigma_\nu)}(\int_K \parallel \mathcal{F}_\nu F(\lambda,k) \parallel_\nu^2 {\rm d}k)^{\frac{1}{2}}\\
&\leq C_\nu |c_\nu(\lambda)| R^{1/2} \parallel f\parallel_{L^2(K,\sigma_\nu)}\parallel F\parallel_{L^2(G,\tau_\nu)} ,
\end{split}\end{align*}
by the restriction Fourier theorem. Taking the supermum over all $F$ with $\parallel F\parallel_{L^2(G,\tau_\nu)}=1$, the corollary follows.
\end{proof}
\section{Asymptotic expansion for the Poisson transform}
In  this section we give an asymptotic expansion for the Poisson transform. We first start by establishing some intermediate results.\\
Let \(L^{2}_{\lambda}(K, \sigma_\nu)\) denote the finite linear span of the functions 
$$
f^g_{\lambda, v}:k\longmapsto f^g_{\lambda,v}(k)={\rm e}^{(i\lambda-\rho)H(g^{-1}k)}\tau_\nu^{-1}(\kappa(g^{-1}k))v, \quad g\in G, v\in V_\nu.
$$
\begin{lem}
For $\lambda \in \mathbb{R}\setminus \{0\}, \nu \in {\N}$  the space $ L^{2}_{\lambda}(K,\sigma_\nu)$ is a dense subspace of $L^{2}(K,\sigma_\nu)$.
\end{lem}
\begin{proof} As $\lambda\in \mathbb{R}\setminus\{0\}$, the density is just a reformulation of the injectivity  of the Poisson transform $\mathcal{P}_{\nu,\lambda}$.
\end{proof}
\begin{lem}\label{dense}
Let $\lambda \in \R\setminus\{0\}, \nu \in \mathbb{N}$. Then  there exists a unique unitary isomorphism $U^\nu_\lambda$ on $L^{2}(K,\sigma_\nu)$ such that :
\begin{equation*}  
U^\nu_\lambda\,f^g_{\lambda,v}= f^g_{-\lambda,v}, \quad g\in G.
\end{equation*}
Moreover, for $f_1, f_2 \in L^{2}(K, \sigma_\nu)$, we have $\displaystyle\mathcal{P}^\nu_\lambda F_{1}=\mathcal{P}^\nu_{-\lambda}F_2$  if and only if $U^\nu_\lambda F_1=F_2$ ( i.e. $\displaystyle U^\nu_\lambda=(\mathcal{P}^\nu_{-\lambda})^{-1} \circ \mathcal{P}^\nu_\lambda$).
 \end{lem}
\begin{proof} 
The proof is the same as in \cite{Ka} (see also Lemma 5.2 in\cite{BIA}) so we omit it.
\end{proof}
We now introduce the function space \(B^*(G,\tau_\nu)\) on G, consisting of functions $F$ in $ L^2_{loc}(G,\tau_{\nu})$ satisfying
\begin{align*}
\parallel F\parallel_{B^*(G,\tau_\nu)}=\sup_{j\in \N}[2^{-\frac{j}{2}}\int_{A_j}\parallel F(g)\parallel_\nu^2\, {\rm d}g_K]<\infty,
\end{align*}
where \(\displaystyle A_{0}=\{g\in G; d(0,g.0)<1\}\,  \textit{and}\, \,  \displaystyle A_{j}=\{g \in G; 2^{j-1}\leq d(0,g.0)< 2^{j}\}\),  for  \(j\geq 1\).\\
One could easily show that $
\parallel F\parallel_{B^*(G,\tau_\nu)}\leq \parallel F\parallel_\ast \leq 2\parallel F\parallel_{B^*(G,\tau_\nu)}$.\\ 
We define an equivalent relation on \(B^*(G,\tau_\nu)\). For $F_1, F_2 \in B^*(G,\tau_\nu)$ we write $F_1\simeq F_2$ if 
$$
\lim_{R\rightarrow +\infty}\frac{1}{R}\int_{B(R)}\parallel  F_1(g)-F_2(g)\parallel_\nu^2\, {\rm d}g_K =0.
$$
Note that  by using the polar  decomposition we see that $F_1\simeq F_2$ if 
\begin{align*}
\lim_{R\rightarrow +\infty}\frac{1}{R}\int_0^R\int_K\parallel F_1(k{\rm e}^{tH)})-F_2(k{\rm e}^{tH)})\parallel_\nu^2 \,{\rm d}k\,\Delta(t){\rm d}t\ =0.
\end{align*}
We now state the main result of this section
\begin{theo}\label{asympt} Let $\nu \in {\N}, \lambda \in \mathbb{R}\setminus\{0\}$. For $f\in L^{2}(K,\sigma_\nu)$ we have the following asymptotic expansions for the Poisson transform in $B^*(G,\tau_\nu)$
\begin{align}\label{asympt1}
\hspace{1,5cm} \mathcal{P}^\nu_{\lambda}f(x)\simeq \tau_\nu^{-1}(k_2(x))    [\mathbf{c}_\nu(\lambda){\rm e}^{(i\lambda-\rho)(A^+(x)}f(k_1(x))+\mathbf{c}_\nu(-\lambda)e^{(-i\lambda-\rho)(A^+(x))}U^\nu_\lambda f(k_1(x))],
\end{align}
where \(x=k_1(x){\rm e}^{A^+(x)}k_2(x)\).
\end{theo}
Most of the proof of the above theorem consists in proving the following Key Lemma, giving the asymptotic expansion for the translates of the $\tau_\nu$-spherical function.\\

\textbf{KEY LEMMA.}
For $\lambda\in \mathbb{R}\setminus\{0\}, g\in G$ and $v\in V_\nu$, we have the following asymptotic expansion in $B^*(G,\tau_\nu)$
\begin{align*}
\Phi_{\nu,\lambda}(g^{-1}x).\,v\simeq \tau_\nu^{-1}(k_2(x))\sum_{s\in \{\pm 1\}}\mathbf{c}_\nu(s\lambda){\rm e}^{(is\lambda-\rho)A^+(x)}f^g_{s\lambda,v}(k_1(x)),
\end{align*}
$x=k_1(x){\rm e}^{A^+(x)}k_2(x)$.\\

\textbf{Proof of Theorem \ref{asympt}.}
We first note that both side of \eqref{asympt1} depend continuously on $f\in L^2(K,\sigma_\nu)$. This can be proved in the same manner as in \cite{BIA}. Therefore we only have to prove  that the asymptotic expansion \eqref{asympt1} holds for $f\in L_\lambda^2(K,\sigma_\nu)$. Let $\displaystyle f= f_{\lambda,v}^g$. Then according to [\cite{com}, Proposition 3.3], we have
\begin{align*}
\mathcal{P}^\nu_\lambda f(x)=\Phi_{\nu,\lambda}(g^{-1}x)v.
\end{align*}
The theorem follows from the Key lemma.\\
As a consequence of Theorem \ref{asympt} we obtain the following result giving the behaviour of the Poisson integrals.
\begin{pro}
\begin{enumerate}
\item For any \(f\in L^2(K,\sigma_\nu)\) we have the Plancherel-Poisson formula
\begin{align}\label{E2}
\lim_{R\rightarrow +\infty}\frac{1}{R}\int_{B(R)}\parallel  \mathcal{P}_\lambda^\nu f(g)\parallel_\nu^2\,{\rm d}g_K=2\mid \mathbf{c}_\nu(\lambda)\mid^2\, \parallel f\parallel^2_{L^2(K,\sigma_\nu)}
\end{align}
\item Let \( \nu\in {\N}\). There exists a positive constant \(C_{\nu}\) such that for any \(\lambda\in {\R}\setminus\{0\}\),  we have
\begin{align}\label{E5}
C_{\nu}^{-1}\mid \mathbf{c}_\nu(\lambda)\mid \,  \parallel f\parallel_{L^2(K,\sigma_\nu)}\leq \parallel \mathcal{P}_\nu^\lambda f\parallel_\ast\leq C_{\nu}\mid \mathbf{c}_\nu(\lambda)\mid\,  \parallel f\parallel_{L^2(K,\sigma_\nu)},
\end{align}
for every \(f\in L^2(K,\sigma_\nu)\).
\end{enumerate}
\end{pro}
\begin{proof}
\begin{enumerate}
\item  We define for \(f\in L^2(K,\sigma_\nu)\)
\begin{align*}
S^\nu_\lambda f(x):=\tau_\nu^{-1}(k_2(x))    [\mathbf{c}_\nu(\lambda){\rm e}^{(i\lambda-\rho)(A^+(x)}f(k_1(x))+\mathbf{c}_\nu(-\lambda)e^{(-i\lambda-\rho)(A^+(x))}U^\nu_\lambda f(k_1(x))],
\end{align*}
$x=k_1(x){\rm e}^{A^+(x)}k_2(x)$.\\ By the unitarity of $U_\lambda$, we have 
\begin{align*}\begin{split}
\dfrac{1}{R} \int_{B(R)}\|S^\nu_\lambda f(g)\|^{2}dg_K
 &= 2|\mathbf{c}_\nu(\lambda)|^{2}\|f\|^2_{L^{2}(K,\tau_{\nu})}\left(\frac{1}{R}\int_0^R {\rm e}^{-2\rho t}\Delta(t){\rm d}t\right)\\
 & + 2|\mathbf{c}_\nu(\lambda)|^{2}\Re \left(<f,U_\lambda f>_{L^2(K,\sigma_\nu)}\frac{1}{R}\int_{0}^{R}e^{2(i\lambda-\rho)t}\Delta(t)dt\right).
\end{split}\end{align*}
From $\displaystyle \lim_{R\rightarrow +\infty}\frac{1}{R}\int_0^R {\rm e}^{-2\rho t}\Delta(t){\rm d}t=1$, and $\displaystyle\lim_{R\rightarrow +\infty}\frac{1}{R}\int_0^R e^{2(i\lambda-\rho)t}
\Delta(t){\rm d}t=0$, we deduce that
\begin{align}\label{E4}
\lim_{R\rightarrow +\infty}\frac{1}{R}\int_{B(R)}\parallel S^\nu_\lambda f(g)\parallel_\nu^2{\rm d}g_K=2\mid \mathbf{c}_\nu(\lambda)\mid^2\parallel f\parallel_{L^2(K,\sigma_\nu)}^2.
\end{align}
Next write
\begin{align*} \frac{1}{R} \int_{B(R)} \parallel\mathcal{P}_\lambda^\nu f(g)\parallel_\nu^2 dg_K &=\frac{1}{R} \int_{B(R)}( \parallel S^\nu_\lambda f(g)\parallel_\nu^2+\parallel\mathcal{P}_\lambda^\nu f(g)-S^\nu_\lambda f(g)\parallel_\nu^2\\
&+  2Re [< \mathcal{P}^\nu_\lambda f(g)-S^\nu_\lambda f(g),S^\nu_\lambda f(g)>])  dg_K.
\end{align*}
The estimate \eqref{E2} then follows from  \eqref{E4}, Theorem \ref{asympt} and the Cauchy-Schwarz inequality.
\item The right hand side of the estimate (\ref{E5}) has already been proved, see corollary \ref{nc}.\\The left hand side of the estimate (\ref{E5}) obviously follows from the estimate (\ref{E2}). This finishes the proof of the proposition.
\end{enumerate}
\begin{rem}
Let \(f_1,f_2\in L^2(K,\sigma_\nu)\). Then using the polarization identity as well as  the estimate (\ref{E2}), we get
\begin{align}\label{E6}
\lim_{R\rightarrow +\infty}\frac{1}{R}\int_{B(R)}<\mathcal{P}_\lambda^\nu f_1(g),\mathcal{P}_\lambda^\nu f_2(g)>_\nu{\rm d}g_K=2\mid \mathbf{c}_\nu(\lambda)\mid^2<f_1,f_2>_{L^2(K, \sigma_\nu)}
\end{align}
\end{rem}
\end{proof}
\section{Proof of the main results}
In this section we shall prove  Theorem \ref{range} on  the $L^2$-range of the vector Poisson transform  and Theorem \ref{spectre} characterizing the image $\mathcal{Q}_\lambda^\nu(L^2(G,\tau_\nu)$.
\subsection{The $L^2$-range of the Poisson transform}
We  first recall some results of harmonic analysis on the homogeneous vector bundle \(K\times_M V_\nu\) associated to the representation \(\sigma_\nu\) of \(M\).\\  
Let \(\widehat{K}\) be the unitary dual of \(K\). For \(\delta\in \widehat{K}\) let \(V_\delta\) denote a representation space of \(\delta\) with $d_\delta=\dim V_\delta$. We denote by  $\widehat{K}(\sigma_\nu)$  the set of $\delta\in \widehat{K}$ such that $\sigma_\nu$ occurs in $\delta\mid_M$ with multiplicity $m_\delta>0$.\\ 
 The decomposition of $L^2(K,\sigma_\nu)$ under $K$ (the group $K$ acts by left translations on this space) is given by the Frobenius reciprocity law 
$$
 L^2(K,\sigma_\nu)=\bigoplus_{\delta\in \widehat{K}(\sigma_\nu)} V_\delta\otimes Hom_M(V_\nu,V_\delta),
$$
where $v\otimes L$, for $v\in V_\delta, L\in Hom_M(V_\nu,V_\delta)$ is identified with the function $(v\otimes L)(k)=L^\ast(\delta(k^{-1})v)$, where $L^\ast$ denotes the adjoint of $L$.\\
For each \(\delta\in \widehat{K}(\sigma_\nu)\) 
let \((L_j)_{j=1}^{m_\delta}\) be an orthonormal basis of \(Hom_M(V_\nu,V_\delta)\) with respect to the inner product\\ \(\displaystyle <L_1,L_2>=\frac{1}{\nu+1}Tr(L_1L_2^\ast)\).\\
Let \(\{v_1,\cdots,v_{d_\delta}\}\) be an orhonormal basis of \(V_\delta\). Then 
$$
f_{ij}^\delta:k\rightarrow \sqrt{\frac{d_\delta}{\nu+1}} L_i^\ast\delta(k^{-1})v_j, \quad 1\leq i\leq m_\delta, \quad 1\leq j \leq d_\delta, \quad  \delta\in \widehat{K}(\sigma)
$$
form an orthonormal basis of \(L^2(K,\sigma_\nu)\).\\
For \(f\in L^2(K,\sigma_\nu)\) we have the Fourier series expansion \(\displaystyle f(k)=\sum_{\delta\in \widehat{K}(\sigma)}\sum_{i=1}^{m_\delta}\sum_{j=1}^{d_\delta}a^{\delta}_{ij}f_{ij}^\delta(k)\) with 
$$ \displaystyle \parallel f\parallel_{L^2(K,\sigma)}^2=\sum_{\delta\in \widehat{K}(\sigma)}\sum_{i=1}^{m_\delta}\sum_{j=1}^{d_\delta}\mid a^{\delta}_{ij}\mid^2.
$$
We define for \(\delta\in\widehat{K}(\sigma)\) and \(\lambda\in {\C}\), the generalized Eisenstein integral
\begin{equation*}
\Phi^L_{\lambda,\delta}(g)=\int_K {\rm e}^{-(i\lambda+\rho)H(g^{-1}k)}\tau_\nu(\kappa(g^{-1}k))L^\ast\delta(k^{-1}){\rm d}k, \quad L\in Hom_M(V_\nu,V_\delta).
\end{equation*}
It is easy to see that  \(\Phi^L_{\lambda,\delta}\) satisfies the following identity
\begin{equation*}
\Phi^L_{\lambda,\delta}(k_1gk_2)=\tau_\nu(k_2^{-1})\Phi^L_{\lambda,\delta}(g)\delta(k_1^{-1}), \quad k_1,k_2\in K,\, g\in G.
\end{equation*}
We now prove  an asymptotic  estimate for the generalized Eisenstein integrals.
\begin{pro}
Let \(\nu\in {\N}, \lambda\in {\R}\setminus\{0\}\). Then for \(\delta\in \widehat{K}(\sigma_\nu), T, S\in Hom_M(V_\nu,V_\delta)\) we have
\begin{align}\label{asympt2}
\lim_{R\rightarrow +\infty}\frac{1}{R}\int_{B(R)} \textit{Tr}\left(\Phi_{\lambda,\delta}^T(g)^\ast \Phi_{\lambda,\delta}^S(g)\right){\rm d}g_K=2\mid \mathbf{c}_\nu(\lambda)\mid^2 \textit{Tr}(TS^\ast).
\end{align}
\end{pro}
\begin{proof}
By definition we have
\begin{align*}
\lim_{R\rightarrow +\infty}\frac{1}{R}\int_{B(R)}\textit{Tr}\left(\Phi_{\lambda,\delta}^T(g)^\ast \Phi_{\lambda,\delta}^S(g)\right){\rm d}g_K =\sum_{j=1}^{d_\delta}\lim_{R\rightarrow +\infty}\frac{1}{R}\int_{B(R)}< \Phi_{\lambda,\delta}^S(g)v_j,\Phi_{\lambda,\delta}^T(g)v_j>_\nu {\rm d}g_K
\end{align*}
Noting that $\Phi_{\lambda,\delta}^T(g)v_j$ is the Poisson transform of the function $k\mapsto L^*\delta(k^{-1})v_j$ and using \eqref{E6}, we get
\begin{align*}
\lim_{R\rightarrow +\infty}\frac{1}{R}\int_{B(R)}\textit{Tr}\left(\Phi_{\lambda,\delta}^T(g)^\ast \Phi_{\lambda,\delta}^S(g)\right){\rm d}g_K=2\mid \mathbf{c}_\nu(\lambda)\mid^2 \sum_{j=1}^{d_\delta}\int_K <S^\ast\delta(k^{-1})v_j,T^\ast\delta(k^{-1})v_j>_\nu {\rm d}k.
\end{align*}
Hence  Schur Lemma lead us to conclude that 
$\displaystyle \lim_{R\rightarrow +\infty}\frac{1}{R}\int_{B(R)}\textit{Tr}\left(\Phi_{\lambda,\delta}^T(g)^\ast \Phi_{\lambda,\delta}^S(g)\right){\rm d}g_K = 2\mid \mathbf{c}_\nu(\lambda)\mid^2 \textit{Tr}(TS^\ast)$, and the proof is finished.
\end{proof}
\begin{rem}
Noting that 
$$
Tr(\left(\Phi_{\lambda,\delta}^T(g)^\ast \Phi_{\lambda,\delta}^S(g)\right)=Tr(\left(\Phi_{\lambda,\delta}^T(a_t)^\ast \Phi_{\lambda,\delta}^S(a_t)\right),\quad g=k_1\,a_t\,k_2,
$$
it follows from \eqref{asympt2} that 
\begin{align}\label{asympt3}
\lim_{R\rightarrow +\infty}\frac{1}{R}\int_0^R Tr\left(\Phi_{\lambda,\delta}^T(a_t)^\ast \Phi_{\lambda,\delta}^S(a_t)\right)\Delta(t){\rm d}t = 2\mid \mathbf{c}_\nu(\lambda)\mid^2 \textit{Tr}(TS^\ast).
\end{align}
\end{rem}
\textbf{Proof of Theorem \ref{range}.}\\
(i) The estimate \eqref{E5} implies that the Poisson transform $\mathcal{P}_\lambda^\nu$ maps $L^2(K,\sigma_\nu)$ into $\mathcal{E}_\lambda(G,\tau_\nu)$ and that the  estimate  \eqref{Poisson estimates} holds.\\
(ii) We  now prove that the Poisson transform maps $L^2(K,\sigma_\nu)$ onto $\mathcal{E}^2_\lambda(G,\tau_\nu)$. Let \(F\in \mathcal{E}^2_\lambda(G,\tau_\nu)\). Since $\lambda\in \mathbb{R}\setminus\{0\}$, we know  by Theorem \ref{O} that there exists a hyperfunction $f\in C^{-\omega}(K,\sigma_\nu)$ such that $F=\mathcal{P}_\lambda^\nu f$.\\
Let
$ 
\displaystyle f=\sum_{\delta\in \widehat{K}(\sigma)}\sum_{j=1}^{d_\delta}\sum_{i=1}^ {m_\delta} a^\delta_{i j}f_{i j}^\delta
$,  be the Fourier series expansion of $f$. Then we have
\begin{align*}
 F(g)=\sum_{\delta\in \widehat{K}(\sigma)}\sqrt{\frac{d_\delta}{\nu+1}}\sum_{j=1}^{d_\delta}\sum_{i=1}^ {m_\delta} a^\delta_{i j}\Phi_{\lambda,\delta}^{L_i}(g)v_j \quad \textit{in}\quad  C^\infty(G,V).
\end{align*}
By the Schur relations, we have
\begin{align*}
\int_K <\Phi_{\lambda,\delta}^{L_i}(ka_t)v_j, \Phi_{\lambda,\delta'}^{L_m}(ka_t)v_n>_\nu\, {\rm d}k=\left\lbrace
\begin{array}{ll}
0&\,  \text{if } \delta\nsim \delta'
\\
\frac{1}{d_\delta} Tr(\Phi_{\lambda,\delta^{'}}^{L_m}(a_t))^*\Phi_{\lambda,\delta}^{L_i}(a_t)<v_j,v_n>_\nu
\quad\quad\text{if } \quad \delta'=\delta
\end{array}\right.
\end{align*}
Therefore
\begin{align*}
\int_K \parallel F(ka_t)\parallel_\nu^2 {\rm d}k&=\frac{1}{\nu+1}\sum_{\delta\in \widehat{K}(\sigma)}\sum_{j=1}^{d_\delta}\sum_{1\leq i,j\leq m_\delta} a^\delta_{i j}\overline{a^\delta_{mj}}Tr [(\Phi_{\lambda,\delta}^{L_m}(a_t))^\ast\Phi_{\lambda,\delta}^{L_i}(a_t)]\\
&=\frac{1}{\nu+1}\sum_{\delta\in \widehat{K}(\sigma)}\sum_{j=1}^{d_\delta}Tr\left[\sum_{1\leq i,m\leq m_\delta} (a^\delta_{mj} \Phi_{\lambda,\delta}^{L_m}(a_t))^\ast (a^\delta_{i j}\Phi_{\lambda,\delta}^{L_i}(a_t)\right]\\
&=\frac{1}{\nu+1}\sum_{\delta\in \widehat{K}(\sigma)}\sum_{j=1}^{d_\delta}\parallel \sum_{i=1}^{m_\delta}a^\delta_{i j} \Phi_{\lambda,\delta}^{L_i}(a_t)\parallel^2_{HS},
\end{align*}
Let $\Lambda$ be a finite subset in $\widehat{K}(\sigma)$. Since  $\parallel F\parallel_\ast <\infty$, it follows that, for any $R>1$ we have
\begin{align*}
\infty>\parallel F\parallel^2_\ast\geq \frac{1}{\nu+1}\sum_{\delta\in\Lambda}\sum_{j=1}^{d_\delta}\frac{1}{R}\int_0^R\parallel \sum_{i=1}^{m_\delta}a^\delta_{i j} \Phi_{\lambda,\delta}^{L_i}(a_t)\parallel^2_{HS}\, \Delta(t)\, {\rm d}t
\end{align*}
By   \eqref{asympt3} we have
\begin{align*}
\lim_{R\rightarrow \infty}\frac{1}{R}\int_0^R\parallel \sum_{i=1}^{m_\delta}a^\delta_{i j} \Phi_{\lambda,\delta}^{L_i}(a_t)\parallel^2_{HS}\, \Delta(t)\, {\rm d}t &=\lim_{R\rightarrow \infty}\sum_{1\leq i,m\leq m_\delta}a^\delta_{i j}\overline{a^\delta_{m j}}\,\frac{1}{R}\int_0^R Tr [(\Phi_{\lambda,\delta}^{L_m}(a_t))^\ast\Phi_{\lambda,\delta}^{L_i}(a_t)]\, \Delta(t){\rm d}t\\
&=2\mid \mathbf{c}_\nu(\lambda)\mid^2\,\sum_{1\leq i,m\leq m_\delta} a^\delta_{i j}\overline{a^\delta_{m j}}Tr(L_i L_m^*)\\
&=2(\nu+1)\mid \mathbf{c}_\nu(\lambda)\mid^2\sum_{i=1}^{m_\delta}\mid a^\delta_{i j}\mid^2.
\end{align*}
Thus 
$
\displaystyle\infty>\parallel F\parallel^2_\ast\geq \mid \mathbf{c}_\nu(\lambda)\mid^2 \sum_{\delta\in\Lambda} \sum_{j=1}^{d_\delta} \sum_{i=1}^{m_\delta}\mid a^\delta_{i j}\mid^2
$.
Since $\Lambda$ is arbitrary, it follows that 
\begin{align*}
\mid \mathbf{c}_\nu(\lambda)\mid^2 \sum_{\delta\in \widehat{K}(\sigma)} \sum_{j=1}^{d_\delta} \sum_{i=1}^{m_\delta}\mid a^\delta_{i j}\mid^2\leq \parallel F\parallel^2_\ast.
\end{align*}
This shows that $f\in L^2(K,\sigma_\nu)$ with $\mid \mathbf{c}_\nu(\lambda)\mid \parallel f\parallel_{L^2(K,\sigma_\nu)}\leq \parallel \mathcal{P}^\nu_\lambda f\parallel_*$ and the proof of the theorem is completed.
\subsection{The $L^2$-range of the generalized spectral projections}
We now proceed to the poof of the second main result of this paper.\\
\textbf{Proof of Theorem \ref{spectre}.}\\
Let $F\in L^2_c(G,\tau_\nu)\cap C^\infty(G,\tau_\nu)$. It follows from the definition ( see  \eqref{specv})  that the operator $\mathcal{Q}^\nu_\lambda$ may be written as 
\begin{align}\label{link}
\mathcal{Q}^\nu_\lambda F(g)=\mid \mathbf{c}_\nu(\lambda)\mid^{-2}\mathcal{P}^\nu_\lambda(\mathcal{F}_\nu F(\lambda,.))(g).
\end{align}
Using Theorem \ref{range} we deduce that
\begin{align*}
\sup_{R>1}\frac{1}{R}\int_{B(R)}\parallel \mathcal{Q}^\nu_\lambda F(g)\parallel_\nu^2\, {\rm d}g_K\leq C_\nu\,\mid \mathbf{c}_\nu(\lambda)\mid^{-2}\int_K \parallel \mathcal{F}_\nu F(\lambda,k)\parallel_\nu^2\,{\rm d}k.
\end{align*}
The above inequality and the Plancherel formula \eqref{plancherel2} imply
\begin{align*}
\int_0^\infty (\sup_{R>1}\frac{1}{R}\int_{B(R)}\parallel \mathcal{Q}^\nu_\lambda F(g)\parallel_\nu^2\, {\rm d}g_K)\, {\rm d}\lambda &\leq C_\nu\int_0^\infty\int_K \parallel \mathcal{F}_\nu F(\lambda,k)\parallel_\nu^2 \mid \mathbf{c}_\nu(\lambda)\mid^{-2}\,{\rm d}k\, {\rm d}\lambda\\
&\leq C_\nu \parallel F\parallel^2_{L^2(G,\tau)}.
\end{align*}
This prove the right hand side of the inequality \eqref{spectral2}.\\
From \eqref{link} and \eqref{estimate00} we have 
\begin{align*}
\lim_{R\rightarrow \infty}\frac{1}{R}\int_{B(R)}\parallel \mathcal{Q}^\nu_\lambda F(g)\parallel_\nu^2\, {\rm d}g_K=2\mid \mathbf{c}_\nu(\lambda)\mid^{-2}\int_K\parallel\mathcal{F}_\nu F(\lambda,k)\parallel_\nu^2 {\rm d}k,
\end{align*} 
and since for all $R>1$
\begin{align*}
\frac{1}{R}\int_{B(R)}\parallel \mathcal{Q}^\nu_\lambda F(g)\parallel^2\, {\rm d}g_K\leq C_\nu\mid \mathbf{c}_\nu(\lambda)\mid^{-2}\int_K\parallel\mathcal{F}_\nu F(\lambda,k)\parallel^2 {\rm d}k, \quad \textit{a.e.} \, \,  \lambda\in (0,\infty),
\end{align*}
we may apply the Lebesgue's dominated convergence theorem to get
\begin{align*}
\lim_{R\rightarrow \infty}\int_0^\infty \left(\frac{1}{R}\int_{B(R)}\parallel  \mathcal{Q}^\nu_\lambda F(g)\parallel_\nu^2\, {\rm d}g_K\right){\rm d}\lambda= 2 \parallel F\parallel^2_{L^2(G,\tau_\nu)}.
\end{align*}
It follows from the above equality that $$
 C \parallel F\parallel^2_{L^2(G,\tau_\nu)}\leq \int_0^\infty(\sup_{R>1}\int_{B(R)}\parallel  \mathcal{Q}^\nu_\lambda F(x)\parallel^2\, {\rm d}x)\, {\rm d}\lambda.
 $$
This complete the proof of the inequality \eqref{spectral2}.\\
We now prove that $ \mathcal{Q}^\nu_\lambda $ maps $L^2_c(G,\tau_\nu)$ onto $\mathcal{E}^2_\lambda(G,\tau_\nu)$.
Let $F_\lambda\in \mathcal{E}^2_\lambda(G,\tau_\nu)$. Then we have
\begin{align*}
\sup_{R>1}\frac{1}{R}\int_{B(R)}\parallel F_\lambda(g)\parallel_\nu^2\, {\rm d}g_K<\infty, \quad \textit{for a.e.} \quad \lambda\in \,(0,\infty).
\end{align*}
By Theorem \ref{range}, there exists $f_\lambda\in L^2(K,\sigma_\nu)$ such that $F_\lambda(g)=\mid \mathbf{c}_\nu(\lambda)\mid^{-2}\mathcal{P}^\nu_\lambda f_\lambda(g)$ with
\begin{align*}
\sup_{R>1}\frac{1}{R}\int_{B(R)}\parallel F_\lambda(g)\parallel_\nu^2\, {\rm d}g_K\geq C_\nu^{-1}\mid \mathbf{c}_\nu(\lambda)\mid^{-2}\int_K \parallel f_\lambda(k)\parallel^2\,{\rm d}k
\end{align*} 
Integrating the both side of the above inequality over $(0,\infty)$, we get
\begin{align*}
\infty>\parallel F_\lambda\parallel_*^2\geq C_\nu^{-1}\int_O^\infty\int_K  \parallel f_\lambda(k)\parallel_\nu^2\,\mid \mathbf{c}_\nu(\lambda)\mid^{-2}{\rm d}k\, {\rm d}\lambda.
\end{align*}
It now follows from Theorem \ref{fourier}, that there exists $F\in L^2_c(G,\tau_\nu)$ such that $\mathcal{F}_\nu F(\lambda,k)= f_\lambda(k)$. \\Henceforth $F_\lambda(g)=\mid \mathbf{c}_\nu(\lambda)\mid^{-2}\mathcal{P}_\lambda^\nu (\mathcal{F}_\nu F(\lambda,.)(g)$. This finishes the proof of Theorem \ref{spectre}. 
\section{Proof of the Key Lemma}
In this section we prove the Key Lemma of this paper.
To this end we need to establish some auxiliary results. We first prove an asymptotic formula for the $\tau_\nu$-spherical function.
\begin{pro}\label{asymp3b}
Let $\lambda\in \mathbb{R}\setminus\{0\}$. For any $v\in V_\nu$ we have 
\begin{align}\label{asymp3}
\Phi_{\nu,\lambda}(g).\,v\simeq\sum_{s\in \{\pm 1\}}\mathbf{c}_\nu(s\lambda){\rm e}^{(is\lambda-\rho)A^+(g)}\tau_\nu^{-1}(\kappa_1(g)\kappa_2(g)).\,v,
\end{align}
$ g=\kappa_1(g){\rm e}^{A^+(g)}\kappa_2(g)$
\end{pro}
\begin{proof}
Since $\Delta(t)\leq 2^3{\rm e}^{2\rho\,t}$, we get
\begin{align*}\begin{split}
\frac{1}{R}\int_{B(R)} \parallel  {\rm e}^{(i\lambda-\rho) A^+(g)}\tau_\nu^{-1}(\kappa_1(g)\kappa_2(g)).\,v\parallel^2\, {\rm d}g_K &=\frac{1}{R}\parallel v\parallel^2\int_0^R {\rm e}^{-2\rho\,t}\Delta(t){\rm d}t\\
&\leq 2^3\parallel v\parallel^2.
\end{split}\end{align*}
This shows that the right hand side of \eqref{asymp3} belongs to $B^*(G,\tau_\nu)$.\\
Since $\lambda\in \mathbb{R}\setminus \{0\}$, we may use the identity \eqref{A3} to write  
\begin{align*}\begin{split}
\varphi_{\nu,\lambda}(t)-\sum_{s\in \{\pm 1\}}\mathbf{c} _{\nu}(s\lambda){\rm e}^{(is\lambda-\rho)t}&
=\sum_{s\in \{\pm 1\}}\mathbf{c}_\nu(s\lambda)\left((2\cosh t)^\nu\Psi^{\rho-2,\nu+1}_{s\lambda}(t)-{\rm e}^{(is\lambda-\rho)t}\right)\\
&=\sum_{s\in \{\pm 1\}}\mathbf{c}_\nu(s\lambda){\rm e}^{(is\lambda-\rho)t}\left((1+{\rm e}^{-2t})^\nu {\rm e}^{(\rho+\nu-is\lambda)t}\Psi^{\rho-2,\nu+1}_{s\lambda}(t)-1\right).
\end{split}\end{align*}
It follows from \eqref{A2'} that 
\begin{align*}
\varphi_{\nu,\lambda}(t)-\sum_{s\in \{\pm 1\}}\mathbf{c} _{\nu}(s\lambda){\rm e}^{(is\lambda-\rho)t}=\sum_{s\in \{\pm 1\}}\mathbf{c}_\nu(s\lambda){\rm e}^{(is\lambda-\rho)t}\left( (1+{\rm e}^{-2t})^\nu-1)+{\rm e}^{-2t}E_{s\lambda}(t)\right),
\end{align*}
where $\mid E_{s\lambda}(t)\mid\leq 2^\nu C$ if $t\geq 1$. Therefore 
\begin{align*}
\mid \varphi_{\nu,\lambda}(t)-\sum_{s\in \{\pm 1\}}\mathbf{c} _{\nu}(s\lambda){\rm e}^{(is\lambda-\rho)t}\mid\leq C_{\nu,\lambda} {\rm e}^{-\rho t} {\rm e}^{-2t},
\end{align*}
if $t\geq 1$. This together with 
\begin{align*}
\mid \varphi_{\nu,\lambda}(t)-\sum_{s\in \{\pm 1\}}\mathbf{c} _{\nu}(s\lambda){\rm e}^{(is\lambda-\rho)t}\mid\leq C_{\nu,\lambda}{\rm e}^{-\rho t},
\end{align*}
for $t\in [0,1]$, imply that 
\begin{align*}\begin{split}
\lim_{R\rightarrow \infty}\frac{1}{R}&\int_{B(R)}\parallel \Phi_{\nu,\lambda}(g).\,v-\sum_{s\in \{\pm 1\}}c_\nu(s\lambda)e^{(is\lambda-\rho)A^+(g)}\tau^{-1}(\kappa_1(g)\kappa_2(g)).\,v\parallel_\nu^2\,{\rm d}g_K =\\
&=\parallel v\parallel^2 \lim_{R\rightarrow \infty}\frac{1}{R}\int_0^R\mid \varphi_{\nu,\lambda}(t)-\sum_{s\in \{\pm 1\}}c_\nu(s\lambda)e^{(is\lambda-\rho)t}\mid^2\,\Delta(t)\,{\rm d}t=0,
\end{split}\end{align*}
and the proof is finished.
\end{proof}
\begin{lem}
Let $g\in G,k\in K$ and $t$ a non negative real number . Then we have
\begin{align}\label{perturbe}
0\leq A^+(g^{-1}k\exp(tH))-H(g^{-1}k \exp(tH))\leq \frac{ 1+\mid g.0\mid}{ 1-\mid g.0\mid}{\rm e}^ {-2t},
\end{align}
\end{lem}
\begin{proof}
Let $g^{-1}=\begin{pmatrix}
a&b\\c&d,
\end{pmatrix}
$ and $k==\begin{pmatrix}
u&0\\O&v,
\end{pmatrix}
$, where $a, b, c$ and $d$ are $n\times n, n\times 1, 1\times n$ and $1\times 1$ matrices respectively.\\
A direct computation yields
\begin{align*}
g^{-1}k\exp(tH)=\begin{pmatrix}
\ast&\ast\, \ast\\
c_1&d_1
\end{pmatrix},
\end{align*}
where $c_1=c\,u\begin{pmatrix}
\cosh\,t&0\\
0&I_{n-1}
\end{pmatrix}$ and $d_1=\sinh t\,c ue_1+\cosh t\,dv$.\\
By \eqref{component} we have
$$
{\rm e}^{H(g^{-1}k \exp(tH))}={\rm e}^t\mid c u e_1+d v\mid,
$$
and
$$
{\rm e}^{A^+(g^{-1}k \exp(tH))}=\mid \sinh t\,c u e_1+\cosh t\,d v\mid+(\mid \sinh t\,c u e_1+\cosh t\,d v\mid^2-1)^\frac{1}{2}.
$$
From
\begin{align*}
{\rm e}^{A^+(g^{-1}k \exp(tH))-H(g^{-1}k \exp(tH))}=\frac{{\rm e}^{-t}}{\mid c u e_1+d v\mid}[\mid \sinh t\,c u e_1+\cosh t\,d v\mid+(\mid \sinh t\,c u e_1+\cosh t\,d v\mid^2-1)^\frac{1}{2}],
\end{align*}
together with
\begin{align*}
\mid \sinh t\,c u e_1+\cosh t\,d v\mid+(\mid \sinh t\,c u e_1+\cosh t\,d v\mid^2-1)^\frac{1}{2}&\leq 2 \mid \sinh t\,c u e_1v^{-1}+\cosh t\,d\mid\\
&\leq \mid c u e_1v^{-1}+d\mid {\rm e}^t+\mid d-c u e_1v^{-1}\mid {\rm e}^{-t}
\end{align*}
we deduce that 
\begin{align*}
{\rm e}^{(A^+(g^{-1}k \exp(tH))-H(g^{-1}k \exp(tH))}\leq 1+ \frac{\mid d-c u e_1v^{-1}\mid }{\mid c u e_1v^{-1}+d\mid}{\rm e}^{-2t}.
\end{align*}
Noting that $(g.0)^\ast=-(d^{-1}c)$, and $k.e_1=u e_1v^{-1}$, we get
\begin{align*}\begin{split}
{\rm e}^{(A^+(g^{-1}k \exp(tH))-H(g^{-1}k \exp(tH))}&\leq 1+\frac{\mid 1+<g.0,k.e_1>\mid}{\mid 1-<g.0,k.e_1>\mid}{\rm e}^{-2t}\\
&\leq 1+\frac{1+\mid g.0\mid}{1-\mid g.0\mid}{\rm e}^{-2t},
\end{split}\end{align*}
from which we deduce \eqref{perturbe}, and the proof of the lemma is finished.
\end{proof}
\textbf{Proof of the Key Lemma.}
Since $B^\ast(G,\tau_\nu)$ is $G$-invariant, we may apply Proposition \ref{asymp3b} to get
$$ 
\Phi_{\nu,\lambda}(g^{-1}x)v\simeq \tau^{-1}_\nu(\kappa_1(g^{-1}x)\kappa_2(g^{-1}x)\sum_{s\in\{\pm\}}\mathbf{c}_\nu(s\lambda){\rm e}^{(is\lambda-\rho)A^+(g^{-1}x)}v.
$$
Thus it suffices to show that 
\begin{align}\label{key1}
\tau^{-1}_\nu(\kappa_1(g^{-1}x)\kappa_2(g^{-1}x)\sum_{s\in\{\pm\}}\mathbf{c}_\nu(s\lambda){\rm e}^{(is\lambda-\rho)A^+(g^{-1}x)}v\simeq \tau_\nu^{-1}(k_2(x))\sum_{s\in \{\pm 1\}}\mathbf{c}_\nu(s\lambda){\rm e}^{(is\lambda-\rho)A^+(x)}f^g_{s\lambda,v}(k_1(x)),
\end{align}
Note that 
\begin{align*}
\tau_\nu^{-1}[k_1(g^{-1}k_1(x){\rm e}^{A^+(x)}k_2(x)) k_2(g^{-1}k_1(x){\rm e}^{A^+(x)}k_2(x))]=\tau_\nu^{-1}[k_1(g^{-1}k_1(x){\rm e}^{A^+(x)})k_2(g^{-1}k_1(x){\rm e}^{A^+(x)})k_2(x))],
\end{align*}
$x=k_1(x){\rm e}^{A^+(x)}k_2(x)$.\\
Henceforth \eqref{key1} is equivalent to 
\begin{align}\begin{split}\label{key2}
\tau_\nu^{-1}[k_1(g^{-1}k_1(x){\rm e}^{A^+(x)}) k_2(g^{-1}k_1(x){\rm e}^{A^+(x)})]&\sum_{s\in \{\pm 1\}}\mathbf{c}_\nu(s\lambda){\rm e}^{(is\lambda-\rho)A^+(g^{-1}k_1(x){\rm e}^{A^+(x)})}\,v\\
&\simeq \sum_{s\in \{\pm 1\}}\mathbf{c}_\nu(s\lambda){\rm e}^{(is\lambda-\rho)A^+(x)} f^g_{s\lambda,v}(k_1(x)) 
\end{split}\end{align}
We write the left hand side of \eqref{key2} as 
\begin{align*}
\sum_{s\in \{\pm 1\}}\mathbf{c}_\nu(s\lambda){\rm e}^{(is\lambda-\rho)A^+(x)} f^g_{s\lambda,v}(k_1(x)) +r_g(x)v,
\end{align*}
where
\begin{align}\begin{split}\label{key3}
r_g(x)=&\tau_\nu^{-1}[k_1(g^{-1}k_1(x){\rm e}^{A^+(x)})  k_2(g^{-1}k_1(x){\rm e}^{A^+(x)})]\sum_{s\in \{\pm 1\}}\mathbf{c}_\nu(s\lambda){\rm e}^{(is\lambda-\rho)A^+(g^{-1}k_1(x){\rm e}^{A^+(x)})}\\
&-\sum_{s\in \{\pm 1\}}\mathbf{c}_\nu(s\lambda){\rm e}^{(is\lambda-\rho)[A^+(x)+H(g^{-1}k_1(x))]}\tau_\nu^{-1}(\kappa(g^{-1}k_1(x)), \quad x\in G
\end{split}\end{align}
To finish the proof we show that for each $g\in G$, $r_g\simeq 0$.\\
Noting that 
$$
H(g^{-1}k_1(x){\rm e}^{A^+(x)})=H(g^{-1}k_1(x))+A^+(x),
$$
we rewrite $r_g$ as 
\begin{align*}\begin{split}
r_g(x)&=[\tau_\nu^{-1}(k_1(g^{-1}k_1(x){\rm e}^{A^+(x)})  k_2(g^{-1}k_1(x){\rm e}^{A^+(x)}))-\tau_\nu^{-1}(\kappa(g^{-1}k_1(x))]\sum_{s\in \{\pm 1\}}\mathbf{c}_\nu(s\lambda){\rm e}^{(is\lambda-\rho)H(g^{-1}k_1(x){\rm e}^{A^+(x)})}\\
&+\tau_\nu^{-1}(k_1(g^{-1}k_1(x){\rm e}^{A^+(x)})  k_2(g^{-1}k_1(x){\rm e}^{A^+(x)}))\left(\sum_{s\in \{\pm 1\}}\mathbf{c}_\nu(s\lambda)[{\rm e}^{(is\lambda-\rho)A^+(g^{-1}k_1(x){\rm e}^{A^+(x)})}-{\rm e}^{(is\lambda-\rho)H(g^{-1}k_1(x){\rm e}^{A^+(x)})}]\right)\\
&=:I_g(x)+J_g(x).
\end{split}\end{align*}
Using the following result
\begin{lem}\label{rep}
Let $g=\begin{pmatrix}
a&b\\c&d
\end{pmatrix}\in Sp(n,1)$. Then we have
\begin{align}\label{k1}
\tau_\nu(\kappa_1(g)\kappa_2(g))=\tau_\nu(\frac{d}{\mid d\mid})
\end{align}
\begin{align}\label{k2}
\tau_\nu(\kappa(g))=\tau_\nu(\frac{ce_1+d}{\mid ce_1+d\mid})
\end{align}
\begin{align}\label{k3}
\lim_{R\rightarrow \infty}\tau_\nu(\kappa_1(g\exp(RH))\kappa_2(g\exp(RH)))=\tau_\nu(\kappa(g)).
\end{align}
\end{lem}
we easily see that $I_g v\simeq 0$.\\
We have 
\begin{align*}\begin{split}
J_g(x)=&\tau_\nu^{-1}(k_1(g^{-1}k_1(x){\rm e}^{A^+(x)})  k_2(g^{-1}k_1(x){\rm e}^{A^+(x)})){\rm e}^{(is\lambda-\rho)H(g^{-1}k_1(x){\rm e}^{A^+(x)})}\\
&\sum_{s\in \{\pm 1\}}\mathbf{c}_\nu(s\lambda)\left({\rm e}^{(is\lambda-\rho)[A^+(g^{-1}k_1(x){\rm e}^{A^+(x)})-H(g^{-1}k_1(x){\rm e}^{A^+(x)})]}-1\right).
\end{split}\end{align*}
As $\tau_\nu$ is unitary and using polar coordinates we see that $\displaystyle\frac{1}{R}\int_{B(R)}\parallel J_g(x)v\parallel_\nu^2 \,{\rm d}x_K$ is less than 
\begin{align*}
2\parallel v\parallel^2 \frac{\mid \mathbf{c}_\nu(\lambda)\mid^2}{R}\int_0^R\int_K{\rm e}^{-2\rho H(g^{-1}k{\rm e}^{tH})}\mid {\rm e}^{(is\lambda-\rho)[A^+(g^{-1}k{\rm e}^{tH})-H(g^{-1}k{\rm e}^{tH})]}-1\mid^2\, {\rm d}k\,\Delta(t){\rm d}t.
\end{align*}
From the estimate 
\begin{align*}
\mid {\rm e}^{(is\lambda-\rho)[A^+(g^{-1}k{\rm e}^{tH})-H(g^{-1}k{\rm e}^{tH})]}-1\mid\leq C (\mid\lambda\mid+\rho) \mid A^+(g^{-1}k{\rm e}^{tH})-H(g^{-1}k{\rm e}^{tH})\mid
\end{align*}
together with  \eqref{perturbe} we get 
\begin{align*}
\frac{1}{R}\int_{B(R)}\parallel J_g(x)v\parallel_\nu^2 \,{\rm d}x_K\leq     
\left( C (\mid\lambda\mid+\rho)\frac{1+\mid g.0\mid}{1-\mid g.0\mid}\right)^2\frac{1}{R}\int_0^R\int_K{\rm e}^{-2\rho H(g^{-1}k)}{\rm e}^{-2(\rho+2)t}\Delta(t)\,{\rm d}k\,{\rm d}t. 
\end{align*}
As $\displaystyle \int_K {\rm e}^{-2\rho H(g^{-1}k)}\, {\rm d}k=1$ and $\Delta(t)\leq 2^3 {\rm e}^{2\rho t}$ we obtain
\begin{align*}
\lim_{R\rightarrow \infty}\frac{1}{R}\int_{B(R)}\parallel J_g(x)v\parallel_\nu^2 \,{\rm d}x_K =0.
\end{align*}
This shows that $ J_g\simeq 0$. Therefore we have proved that for each $g\in G, r_g\simeq 0$, as to be shown.\\
It remains to prove Lemma \ref{rep}.\\ 
\textbf{Proof of Lemma \ref{rep}.} If $g=\begin{pmatrix}
a&b\\c&d
\end{pmatrix}=\begin{pmatrix}
u_1&0\\
0&v_1
\end{pmatrix}a_t\begin{pmatrix}
u_2&0\\
0&v_2
\end{pmatrix}$ with respect to the Cartan decomposition $G=KAK$. Then we easily see that $d=\cosh t\, v_1v_2$ and \eqref{k1} follows. Analogously if $g=\begin{pmatrix}
a&b\\c&d
\end{pmatrix}=\begin{pmatrix}
u&0\\
0&v
\end{pmatrix}a_t\,n$ with respect to the Iwasawa decomposition. Then from $g.e_1=\begin{pmatrix}
ae_1+b\\ce_1+d
\end{pmatrix}={\rm e}^t\begin{pmatrix}
u\\v
\end{pmatrix}$ we get ${\rm e}^t\,v=ce_1+d$ and \eqref{k2} follows. \\ 
We have 
$$
g\exp(RH)=\begin{pmatrix}
\ast&\ast\ast\\
\ast\ast\ast&\sinh Re_1+\cosh R d
\end{pmatrix}
$$
Then  \eqref{k1} imply that $\tau_\nu(\kappa_1(g)\kappa_2(g))=\tau_\nu(\frac{\tanh R ce_1+d}{\mid \tanh R ce_1+d\mid})$. Thus $\lim_{R\rightarrow \infty}\tau_\nu(\kappa_1(g)\kappa_2(g))=\tau_\nu(\frac{ce_1+d}{\mid ce_1+d\mid})$. This finishes the proof of Lemma \ref{rep}, and the proof of the Key Lemma is completed.
\section{Appendix}
In this section we collect some results on the Jacobi functions, referring to \cite{Ko} for more details.\\
For $\alpha,\beta,\lambda\in\mathbb{C}; \alpha\neq -1,-2,\cdots$ and $t\in\mathbb{R}$, the Jacobi function is defined by 
\begin{align*}
\phi_\lambda^{(\alpha,\beta)}(t)=\,_2 F_1(\frac{i\lambda+\rho_{\alpha,\beta}}{2},\frac{-i\lambda+\rho_{\alpha,\beta}}{2}; \alpha+1; -\sinh^2 t),
\end{align*}
where $_2 F_1$ is the Gauss hypergeometric function and $\rho_{\alpha,\beta}=\alpha+\beta+1$.\\
The Jacobi function $\phi_\lambda^{(\alpha,\beta)}$ is the unique even smooth function on $\mathbb{R}$ which satisfy $\phi_\lambda^{(\alpha,\beta)}(0)=1$ and the differential equation
\begin{equation}
\tag{A1}
\lbrace\frac{d^2}{dt^2}+[(2\alpha+1)\coth t+(2\beta+1)\tanh t]\frac{d}{dt}+ \lambda^2+\rho_{\alpha,\beta}^2\rbrace \phi_\lambda^{(\alpha,\beta)}(t)=0.
\label{A1}
\end{equation}
For $\lambda\notin -i\mathbb{N}$ another solution $\Psi_\lambda^{\alpha,\beta}$  of \eqref{A1}  such that 
\begin{equation}
\tag{A2}
\Psi_\lambda^{\alpha,\beta}(t)={\rm e}^{(i\lambda-\rho_{\alpha,\beta})t}(1+\circ(1)), \quad \textit{as}\quad  t\rightarrow \infty
\label{A2}
\end{equation} 
is given by 
\begin{align*}
\Psi_\lambda^{\alpha,\beta}(t)=(2\sinh t)^{i\lambda-\rho_{\alpha,\beta}}\,\,_2 F_1(\frac{\rho_{\alpha,\beta}-i\lambda}{2},\frac{\beta-\alpha+1-i\lambda}{2}; 1-i\lambda; -\frac{1}{\sinh^2 t}).
\end{align*}
Moreover there exists a constant $C>0$ such that for all $\lambda\in\mathbb{R}$ and all $t\geq 1$ we have
\begin{align*}
\tag{A2'}
\Psi_\lambda^{\alpha,\beta}(t)={\rm e}^{(i\lambda-\rho_{\alpha,\beta})t}(1+{\rm e}^{-2t}\Theta_\lambda(t)), \quad \textit{with}\quad \mid \Theta_\lambda(t)\mid \leq C.
\label{A2'}
\end{align*}
For $\lambda\notin i\mathbb{Z}$, we have 
\begin{align}
\tag{A3}
\phi_\lambda^{(\alpha,\beta)}(t)=\sum_{s=\pm 1}\mathbf{c}_{\alpha,\beta}(s\lambda)\Psi_{s\lambda}^{\alpha,\beta}(t)
\label{A3}
\end{align}
where 
\begin{align*}
\mathbf{c}_{\alpha,\beta}(\lambda)=\frac{2^{\rho_{\alpha,\beta}-i\lambda}\,\Gamma(\alpha+1)\Gamma(i\lambda)}{\Gamma(\frac{i\lambda+\rho_{\alpha,\beta}}{2})\Gamma(\frac{i\lambda+\alpha-\beta+1}{2})}.
\end{align*}
For $\Re(i\lambda)>0$,  the asymptotic behaviour of $\phi_\lambda^{(\alpha,\beta)}$ as $t\rightarrow \infty$ is then given by
\begin{align}
\tag{A4}
\lim_{t\rightarrow \infty}{\rm e}^{(\rho_{\alpha,\beta}-i\lambda)t}\phi_\lambda^{(\alpha,\beta)}(t)=\mathbf{c}_{\alpha,\beta}(\lambda).
\label{A4}
\end{align}
Let $D_e(\mathbb{R})$ denote the space of even smooth function with compact support on $\mathbb{R}$. For $f\in D_e(\mathbb{R})$, the Fourier-Jacobi transform $\mathcal{J}^{\alpha,\beta}f$ ($\lambda\in \mathbb{C}$) is defined  by
\begin{align}
\tag{A5}
\mathcal{J}^{\alpha,\beta}f(\lambda)=\int_0^\infty f(t)\phi_\lambda^{(\alpha,\beta)}(t)\Delta_{\alpha,\beta}(t)\,{\rm d}t,
\label{A5}
\end{align}
where $\Delta_{\alpha,\beta}(t)=(2\sinh t)^{2\alpha+1}(2\cosh t)^{2\beta+1}$.\\
In the sequel, we assume that $\alpha>-1, \beta\in \mathbb{R}$. 
Then the meromorphic function $\mathbf{c}_{\alpha,\beta}(-\lambda)^{-1}$ has only simple poles for $\Im \lambda\geq 0$ which occur  in the set
\begin{align*}
D_{\alpha,\beta}=\{\lambda_k=i(\mid \beta\mid -\alpha-1-2k); k=0,1,\cdots, \mid \beta\mid -\alpha-1-2k>0\}.
\end{align*}
(If $\mid \beta\mid\leq \alpha+1$, then $D_{\alpha,\beta}$ is empty).\\
The following inversion  and Plancherel formulas for the Jacobi transform hold for every $f\in D_e(\mathbb{R})$:
\begin{align}
\tag{A6}
f(t)=\frac{1}{2\pi}\int_0^\infty (\mathcal{J}^{\alpha,\beta}f)(\lambda)\,\phi_\lambda^{(\alpha,\beta)}(t) \mid \mathbf{c}_{\alpha,\beta}(\lambda)\mid^{-2}\, {\rm d}\lambda
+\sum_{\lambda_k\in D_{\alpha,\beta}}d_k (\mathcal{J}^{\alpha,\beta}f)(\lambda_k)\,\phi_{\lambda_k}^{(\alpha,\beta)}(t),
\label{A6}
\end{align}
\begin{align}
\tag{A6'}
\int_0^\infty \mid f(t)\mid^2\, \Delta(t)\, {\rm d}t=\frac{1}{2\pi}\int_0^\infty \mid(\mathcal{J}^{\alpha,\beta}f)(\lambda)\mid^2\,\mid \mathbf{c}_{\alpha,\beta}(\lambda)\mid^{-2}\, {\rm d}\lambda+\sum_{\lambda_k\in D_{\alpha,\beta}}d_k \mid(\mathcal{J}^{\alpha,\beta}f)(\lambda_k)\mid^2
\label{A6'}
\end{align}
where $d_k=-i\,\textit{Res}_{\lambda=\lambda_k}( \mathbf{c}_{\alpha,\beta}(\lambda)\mathbf{c}_{\alpha,\beta}(-\lambda))^{-1}$, is given explicitly by
\begin{align}
\tag{A7}
d_k=(\beta-\alpha-2k-1)\frac{2^{-2(\alpha+\beta)}\Gamma(\alpha+k+1)\Gamma(\beta-k)}{\Gamma^2(\alpha+1)\Gamma(\beta-\alpha-k)k!}.
\label{A7}
\end{align}

\end{document}